%% file: main.tex
\documentclass[11pt]{article}
\usepackage{graphicx} 
\usepackage[a4paper,width=170mm,top=25mm,bottom=25mm]{geometry}

\usepackage{titlesec}

\titleformat{\section}[block]
  {\normalfont\large\scshape\centering}
  {\thesection.}
  {.5em}
  {}

\titleformat{\subsection}[runin]
  {\normalfont\bfseries}
  {\thesubsection.}
  {.5em}
  {}
  [.]

\titlespacing*{\subsection}
  {0pt}
  {1.5ex plus .5ex minus .2ex}
  {.5em}

\titleformat{\subsubsection}[runin]
  {\normalfont\itshape}
  {\thesubsubsection.}
  {.5em}
  {}
  [.]

\titlespacing*{\subsubsection}
  {0pt}
  {1.5ex plus .5ex minus .2ex}
  {.5em}

\usepackage{amsmath, amsthm}
\usepackage{bm}
\usepackage{bbm}
\usepackage{amsfonts}
\usepackage{amssymb}
\usepackage{tikz-cd}

\usepackage{pdfpages}
\usepackage{figures/tikzit}
\input{figures/style1.tikzstyles}

\usepackage{pgfplots}
\usetikzlibrary{calc}
\pgfplotsset{compat=1.18}
\usepgfplotslibrary{colormaps, external}

\usepackage{extarrows}
\usepackage{mathtools}
\usepackage{quiver}
\usepackage{mathrsfs}

\usepackage{emptypage}
\usepackage{pdfpages}

\usepackage{xcolor}
\definecolor{purple}{rgb}{0.4,0,0.7}
\definecolor{green}{rgb}{0.1,0.2,0.8}

\usepackage[colorlinks,linkcolor = blue, citecolor=blue,urlcolor=blue,bookmarks=false,hypertexnames=true]{hyperref}
\usepackage[capitalise]{cleveref}
\usepackage{aliascnt}

\usepackage{adjustbox}
\usepackage[shortlabels]{enumitem}

\DeclareMathOperator{\spann}{span}
\DeclareMathOperator{\id}{id}
\DeclareMathOperator{\End}{End}
\DeclareMathOperator{\Hom}{Hom}

\DeclareMathOperator{\Tr}{Tr}
\DeclareMathOperator{\tr}{tr}

\DeclareMathOperator{\vect}{\mathsf{Vect}}

\DeclareMathOperator{\Rep}{\mathsf{Rep}}

\DeclareMathOperator{\C}{\mathbb{C}}

\DeclareMathOperator{\Res}{Res}

\DeclareMathOperator{\cala}{\mathcal{A}}
\DeclareMathOperator{\calc}{\mathcal{C}}
\DeclareMathOperator{\cald}{\mathcal{D}}
\DeclareMathOperator{\calo}{\mathcal{O}}

\DeclareMathOperator{\calh}{\mathcal{H}}

\DeclareMathOperator{\wt}{wt}

\DeclareMathOperator{\lmod}{\!\text{-}\mathsf{mod}}
\DeclareMathOperator{\rmod}{\mathsf{mod}\text{-}\!}
\DeclareMathOperator{\lproj}{\!\text{-}\mathsf{proj}}
\DeclareMathOperator{\rproj}{\mathsf{proj}\text{-}\!}
\DeclareMathOperator{\proj}{\mathsf{proj}}

\newtheorem*{theorem*}{Theorem}
\newtheorem{theorem}{Theorem}[section]

\newcommand{\newsharedtheorem}[5]{%
  \newaliascnt{#1}{theorem}%
  \newtheorem{#1}[#1]{#2}%
  \aliascntresetthe{#1}%
  \crefname{#1}{#2}{#5}%
  \Crefname{#1}{#2}{#5}%
}

\crefname{theorem}{Theorem}{Theorems}
\Crefname{theorem}{Theorem}{Theorems}

\newsharedtheorem{lemma}{Lemma}{lemma}{lemmas}{Lemmas}
\newsharedtheorem{proposition}{Proposition}{proposition}{propositions}{Propositions}
\newsharedtheorem{corollary}{Corollary}{corollary}{corollaries}{Corollaries}

\theoremstyle{definition}
\newsharedtheorem{definition}{Definition}{definition}{definitions}{Definitions}
\newsharedtheorem{example}{Example}{example}{examples}{Examples}
\newsharedtheorem{remark}{Remark}{remark}{remarks}{Remarks}
\newsharedtheorem{observation}{Observation}{observation}{observations}{Observations}
\newsharedtheorem{convention}{Convention}{convention}{conventions}{Conventions}

\numberwithin{equation}{section}

\title{One-point functions for \(C_2\)-cofinite VOAs: pseudo-traces and trace spaces of projective modules}

\begin{document}
\author{Max-Niklas Steffen}
\date{}
\maketitle
\vspace{-0.8cm}
\begin{center}
\textit{Fachbereich Mathematik \\
Universität Hamburg \\
Bereich Algebra und Zahlentheorie \\
Bundesstraße 55 \\
D-20146 Hamburg, Germany \\}
\end{center}

\begin{abstract}
We study the space of one-point functions on the torus for a possibly nonrational \(C_2\)-cofinite
vertex operator algebra \(V\) by relating it to a trace object of the subcategory of projective objects in the representation category of \(V\). We identify the dual of the trace space with symmetric functions on the endomorphism algebra \(E\) of a projective generator. Motivated by the Gainutdinov--Runkel conjecture, recently established using different methods by Gui and Zhang, we present a complementary representation-theoretic approach based on Arike--Nagatomo pseudo-traces. In this framework, we prove surjectivity of the Gainutdinov--Runkel map from symmetric functions on \(E\) to one-point functions. Under the additional assumption of separated conformal weights modulo \(\mathbb{Z}\), we also prove injectivity, using projective-cover techniques inspired by Huang.

\end{abstract}
\tableofcontents


\section{Introduction}
The concept of a vertex operator algebra (VOA) was introduced in the 1980s by Borcherds in \cite{Bo}, and Frenkel, Lepowsky and Meurman in \cite{FLM}. In essence, a \textit{vertex operator algebra} is an \(\mathbb{N}\)-graded \(\mathbb{C}\)-vector space \(V\) with a linear map \(Y\), assigning to every element \(v \in V\) an operator-valued formal Laurent series \(Y(v, z)\)
that satisfies algebraic identities encoding weak versions of associativity and commutativity. A VOA carries an action of the Virasoro algebra, induced by a distinguished element in \(V\), referred to as the conformal vector or Virasoro element.

A fundamental modular invariance result in the theory of VOAs was shown by Zhu in \cite{Z} for \textit{rational} and \(C_2\)\textit{-cofinite} (and nonnegatively graded) VOAs \(V\), where the first condition demands a semisimple representation theory and the second imposes finite-dimensionality of a certain quotient of \(V\). Zhu considered the space of one-point functions on the torus \(\calc_1(V)\), a space of maps \(S \colon V \times \mathbb{H} \rightarrow \C\) linear in the first and holomorphic in the second argument, which satisfy algebraic and differential equations with modular form coefficients. The space of one-point functions comes with a natural action of the modular group \(\mathrm{SL}_2(\mathbb{Z})\), and Zhu established a canonical basis of this space given by forming the trace over certain grade-preserving endomorphisms of irreducible \(V\)-modules. 

Dropping the rationality condition, trace functions are not sufficient to span \(\calc_1(V)\). As Miyamoto proved in \cite{Miy}, one solution is to replace trace functions by \textit{pseudo-trace functions}, which are parametrized by symmetric functions on higher Zhu algebras \(A_n(V)\), a family of associative algebras built from \(V\) indexed by the nonnegative integers. Here, \(V\) is required to come with a nontrivial action of the Virasoro algebra on its simple modules. Miyamoto showed that in this nonsemisimple regime, pseudo-trace functions are powerful enough to generate the space of one-point functions. As higher Zhu algebras are, in general, complicated, Arike and Nagatomo formulated pseudo-trace functions in terms of suitable subalgebras of endomorphism algebras of logarithmic \(V\)-modules in \cite{AN}. 

The latter formulation lends itself nicely to approaching the matter from a categorical point of view, a perspective taken by Gainutdinov and Runkel in \cite{GR}. Under additional assumptions on the VOA (\(V\) simple as a module and isomorphic to its contragredient module, \(V_0 = \C\mathbf{1}\)), the authors consider a suitable representation category \(\calc = \Rep(V)\), earlier studied by Huang in \cite{Hua}, and study one-point functions by considering symmetric functions on the endomorphism algebra \(E\) of a projective generator \(G\). The authors conjecture that \(\calc\) is a pivotal, rigid and factorisable finite tensor category (only part of this statement is known). Under these assumptions, the space of symmetric functions \(\mathcal{S}(E)\) is isomorphic to the hom-space \(\calc(\mathbf{1}, L)\), where \(L\) is the Lyubashenko coend. Mapping a symmetric function to the associated pseudo-trace function provides a canonical map \(\Phi_G \colon \mathcal{S}(E) \rightarrow \calc_1(V)\). From a conjectured relation between 3d TQFT and 2d chiral conformal field theory, it is argued in \cite{GR} that \(\calc(\mathbf{1}, L)\) and \(\calc_1(V)\) must be isomorphic. Consequently, the authors conjecture that \(\Phi_G\) is an isomorphism between symmetric functions and one-point functions. 

Recently, the Gainutdinov--Runkel conjecture on \(\Phi_G\) was proved by Gui and Zhang in \cite{GZ} for \(\mathbb{N}\)-graded \(C_2\)-cofinite VOAs, without assuming simplicity of \(V\) or self-contragredience. Precisely, the authors show, using sewing-factorization techniques for conformal blocks, that the pseudo-trace construction encoded by \(\Phi_G\) gives an isomorphism from symmetric functions on a projective generator to the vector space of vacuum torus conformal blocks, in the precise sense given in the paper.

In the present article, we present an alternative take on the questions outlined above, in a language close to that of \cite{GR}. We work in a setting where we require the VOA \(V\) to be \(C_2\)-cofinite
with nontrivial Virasoro action on simples, and do not impose the additional conditions imposed in \cite{GR}. 
By results of \cite{Hua}, \(\mathcal{C} = \Rep(V)\) is a finite abelian category, and therefore admits a projective generator. Instead of \(\calc(\mathbf{1}, L)\), we consider a different categorical object, namely the trace space \(\Tr(\proj(\calc)) = \int^{X \in \proj(\calc)} \calc(X, X)\), a coend of the hom-functor on the subcategory of projective objects. We identify its dual with the space of symmetric functions \(\mathcal{S}(E)\) on the endomorphism algebra \(E\) of a projective generator \(G\) (see \cref{subSecTr}), thereby obtaining a description independent of the choice of \(G\). Note that this yields, under the assumptions of \cite{GR} and depending on the conjectures there, an isomorphism \(\Tr(\proj(\calc))^* \cong \calc(\mathbf{1}, L)\). Such an identification is not available in our setting.

We further investigate the map \(\Phi_G\) defined above. Making use of Arike--Nagatomo pseudo-trace theory in the framework of \cite{GR}, we prove directly that this map surjects onto \(\calc_1(V)\) (\cref{propSC1iso}). Under the additional assumption that the conformal weights of \(V\) are separated modulo \(\mathbb{Z}\), we give an independent proof that \(\Phi_G\) is also injective (\cref{propEmb} and \cref{corIso}), and therefore an isomorphism identifying \(\mathcal{S}(E)\) with \(\mathcal{C}_1(V)\) as conjectured by Gainutdinov and Runkel. The proof is based on constructing a projective generator from projective covers inspired by \cite{Hua}, and embedding its endomorphism algebra into the top part of the generalized eigenspace decomposition of a Zhu algebra \(A_n(V)\) for a sufficiently large integer \(n\). We expect the restriction on the conformal weights to be of technical rather than conceptual nature, and we expect our methods to generalize to the nonseparated case. Summarizing our findings, we have:

\begin{theorem*} (\cref{thmMain}).
    Let \(V\) be a \(C_2\)-cofinite
    VOA such that the Virasoro algebra acts nontrivially on every simple module. Then, there is a canonical epimorphism 
    \begin{align*}
        \Gamma_G \colon \Tr(\proj\Rep(V))^* \rightarrow \calc_1(V),
    \end{align*}
    for every projective generator \(G\) of \(\Rep(V)\), factoring through \(\Phi_G \colon \mathcal{S}(E) \rightarrow \calc_1(V)\), and \(\Gamma_G\) is an isomorphism if \(V\) has separated conformal weights (modulo \(\mathbb{Z}\)).
\end{theorem*}

The paper is based in part on the author's master's thesis \cite{S}. We intend to offer a complementary perspective to \cite{GZ} to show that the Gainutdinov--Runkel conjecture can also be approached by more classical, representation-theoretic means encountered in VOA theory, such as Zhu algebras, pseudo-traces, universal modules and Huang's projective covers. Let us briefly outline the structure of the article: 
In \cref{secPrel}, we recall from the literature the necessary VOA machinery, including the relevant notions of VOA modules and their properties, the space of one-point functions, as well as the theory of higher Zhu algebras. In \cref{secPseudo}, we collect necessary background on pseudo-traces, and state the main modular invariance result of \cite{Miy}. We end this section by restating Miyamoto's result using Arike--Nagatomo pseudo-traces. Our main results, as outlined above, are contained in \cref{secCat}.
\subsubsection*{Acknowledgments}
I thank Ingo Runkel for his guidance and support during my master's thesis, from which much of this work originated. I also thank Hannes Knötzele, Masahiko Miyamoto, Sven Möller, Christoph Schweigert, and Hiroshi Yamauchi for helpful discussions and correspondence. I gratefully acknowledge scholarship support from the Manke-Förderstiftung during my master's studies.
\section{Algebraic Preliminaries}
\label{secPrel}
\subsection{Notational conventions on additive categories}
For a category \(\calc\), we denote by \(\calc(c, c') = \Hom_{\calc}(c, c')\) its set of morphisms from \(c\) to \(c'\), where \(c, c' \in \calc\) are two objects. We will often consider \(\calc\) to be a \(\C\)-linear additive category, in which case the sets \(\calc(c, c')\) are naturally vector spaces. Standard examples that will appear repeatedly are the (abelian) \(\C\)-linear categories \(\vect_{\C}\) of \(\C\)-vector spaces and its subcategory \(\vect_{\C}^{\mathrm{fd}}\) of finite-dimensional vector spaces, the categories \(A\lmod\) and \(\rmod A\) of finitely generated left and right \(A\)-modules, respectively, where \(A\) is a finite-dimensional associative algebra over \(\C\), as well as their additive \(\C\)-linear subcategories \(A\lproj\) and \(\rproj A\) of finitely generated projective objects.

Generalizing the last example, for \(\calc\) an abelian category, we denote by \(\proj(\calc) \subset \calc\) its (additive) subcategory of projective objects. In this notation, one has \(\proj(A\lmod) = A\lproj\), and an analogous relation holds for right modules.

We denote left \(A\)-modules by \(_A W\) and right \(A\)-modules by \(W_A\), but often leave out the subcripts when it is clear from the context whether we are dealing with left or right modules.
\subsection{Vertex operator algebras and modules}
The basic algebraic object we are interested in is a \textit{vertex operator algebra} (VOA), which is a tuple \((V, Y, \mathbf{1}, \omega)\) consisting of an \(\mathbb{N}\)-graded\footnote{We denote by \(\mathbb{N}\) the nonnegative integers, including zero. Other grading conventions exist, the most standard is that of a \(\mathbb{Z}\)-grading.} \(\C\)-vector space \(V = \oplus_{m = 0}^{\infty}V_m\) with \(\dim_{\C}V_m < \infty\), a map \(Y(-,z) \colon V \rightarrow \End_{\C}(V)[[z, z^{-1}]]\) from the VOA to the space of Laurent series with coefficients given by endomorphisms of \(V\), as well as two distinguished vectors, the \textit{vacuum vector} \(\mathbf{1} \in V_0\) and the \textit{Virasoro element} \(\omega \in V_2\). This structure satisfies various nontrivial algebraic relations -- among those, we have \(Y(\mathbf{1}, z) = \id_V\) and, writing \(Y(v, z) = \sum_{n \in \mathbb{Z}}v_{(n)}z^{-n-1}\) and \(L_n = \omega_{(n+1)}\), the operators \(L_n\) for \(n \in \mathbb{Z}\) define a representation of the Virasoro algebra on \(V\).
We do not recall the full list of properties here, as we will hardly work directly with the definition, and refer to, for example, \cite[Definition 3.1.1, Definition 3.1.22]{LL} for further details.
For \(v \in V_m\), we call the integer \(\wt(v) = m\) the \textit{weight} of \(v\). The weight is compatible with the Virasoro action in the sense that \(L_0v = \wt(v) v\) for a homogeneous element \(v \in V\).

There are different notions of \(V\)-modules that appear in the literature. The basic notion is that of a weak \(V\)-module \((M, Y_M)\), where \(M\) is a \(\C\)-vector space and \(Y_M \colon V \rightarrow \End_{\C}(M)[[z, z^{-1}]]\) a map satisfying a list of properties in relation to \(V\). In particular, \(M\) comes with an action of the Virasoro algebra, and we will keep the notation for the modes of \(Y = Y_M\) and the Virasoro modes from above. There are several ways to strengthen the notion of weak \(V\)-modules, for which one imposes various gradings on \(M\) (for additional details, see, for instance, \cite{DLM1, DLM2, AN}).
\begin{definition}
\begin{enumerate}[(i)]
   \item A weak \(V\)-module \(M\) is said to be \textit{admissible} if there is a grading \(M = \bigoplus_{k \in \mathbb{N}}M(k)\) satisfying \( v_{(m)}M(k) \subseteq M(k + \wt v - m - 1)\).
    \item A weak \(V\)-module \(M\) is a \textit{logarithmic \(V\)-module} if there is a grading \(M = \bigoplus_{\lambda \in \mathbb{C}}M_{(\lambda)} \), where \( M_{(\lambda)} = \{ w \in M|(L_0 - \lambda )^kw = 0 \text{ for some } k \in \mathbb{N}\setminus \{0\} \) is a generalized eigenspace of dimension \(\dim_\mathbb{C}M_{(\lambda)} < \infty\) and for each \(\lambda \in \mathbb{C}\) and \(m\) small enough, \(M_{(m + \lambda)} = 0\).
    \item A logarithmic module is called \textit{ordinary} if the generalized eigenspaces are eigenspaces, that is \(L_0 w = \lambda w\) for any \(w \in M_{(\lambda)}\). To distinguish these from the logarithmic modules, we write \(M_{\lambda} = M_{(\lambda)}\) in this case.
\end{enumerate}
\end{definition}
In the above definition, each of the listed notions of \(V\)-modules is stronger than the previous. For this, it suffices to note that every logarithmic module can be turned into an admissible one by defining an appropriate \(\mathbb{N}\)-grading and using the commutator identity \([L_0, v_{(m)}] = (\wt(v) - m - 1)v_{(m)}\), which is true for every weak \(V\)-module. Note that from this identity it also follows that the \textit{zero-mode} \(o(v) = v_{(\wt v - 1)}\) preserves the grading when the module is logarithmic or admissible.

\looseness = -1 A weak \(V\)-module \(M\) is said to be \textit{finitely generated} if there is a finite subset \(T \subseteq M\) so that \(M\) is spanned by elements \(v^1_{(m_1)} \hdots v^k_{(m_k)}t\), where \(v^i \in V\), \(m_i \in \mathbb{Z}\) and \(t \in T\). A \textit{simple} weak \(V\)-module is one that does not have nontrivial submodules. Simple weak \(V\)-modules are generated by a single element.

A \textit{morphism of weak} \(V\)-\textit{modules} is a linear map \(f \colon M_1 \rightarrow M_2\) satisfying \(f(v_{(m)}w) = v_{(m)}f(w)\) for \(v \in V, w \in M_1, m \in \mathbb{Z}\). For admissible \(V\)-modules, we demand additionally that the morphisms be grade-preserving. For logarithmic modules (and thereby ordinary ones), this is automatic. We write \(\Hom_V(M_1, M_2)\) for the relevant space of morphisms.

Every simple ordinary 
\(V\)-module has the form \(M = \oplus_{m \in \mathbb{N}}M_{r + m}\), with \(M_r \neq 0\). The complex numbers \(r \in \C\) that occur in this way are called the \textit{conformal weights} of \(V\) (or lowest weights, alternatively), and we denote the set of conformal weights for a VOA by \(\Lambda\).

\subsection{Finiteness conditions for VOAs}
\label{secFin}
A vertex operator algebra \(V = \oplus_{m \geq 0}V_m\) is called \textit{rational} if every admissible \(V\)-module can be written as a direct sum of simple admissible \(V\)-modules. This is a strong finiteness assumption on \(V\). Rationality together with the assumption of a different finiteness condition, \(C_2\)-\textit{cofiniteness}, was assumed by Zhu in \cite{Z} to obtain his modular invariance result. Miyamoto later generalized this by dropping the rationality requirement (\cite{Miy}) (see \cref{thmMiy}). 

We recall the definition of \(C_n\)-finiteness.
\begin{definition}
Let \(V\) be a vertex operator algebra. We define for \(n \in \mathbb{N}, n \geq 2\) the subspace \(C_n(V) = \spann_{\mathbb{C}}\{u_{(-n)}v | u, v \in V\}\) of \(V\).
Moreover, we set for \(n = 1\)
\begin{align*}
    C_1(V) = \spann_{\mathbb{C}}\{u_{(-1)}v, L_{-1}w| u, v \in V_+, w \in V\}, 
\end{align*}
where \(V_+ = \bigoplus_{m > 0}V_m\). We call the vertex operator algebra \(V\) \(C_n\)-\textit{cofinite} if \(\dim_{\mathbb{C}} V/C_n(V) < \infty\).
\end{definition}
There is a chain of inclusions \(\hdots C_{n+1}(V) \subseteq C_{n}(V) \subseteq \hdots \subseteq C_{2}(V)\). If \(V\) is \textit{of CFT-type}, which says that \(V_0 = \C \mathbf
1\), then \(C_2(V) \subseteq C_1(V)\), but in general, this is not the case. 
Still, because we are assuming VOAs graded by the nonnegative interges, we have the following.
\begin{lemma}
    \label{lemC2C1}
    A (nonnegatively graded) \(C_2\)-cofinite VOA \(V\) is also \(C_1\)-cofinite.
\end{lemma}
\begin{proof}
    Let \(u_{(-2)}v \in C_2(V)\) and assume without loss of generality that \(u, v \in V\) are homogeneous, so \(u \in V_i, v \in V_j\). By \cite[Lemma 3.1]{Li}, \(u_{(-r-s-k)}v \in C_1(V)\) under the assumption that \(k \geq 1\) and \(r, s \in \mathbb{N}\) such that \(L_{(-1)}^ru, L_{(-1)}^sv \in V_+\). Notice that \(L_{-1} = \omega_{(0)}\) shifts the grading by 1 in the positive direction. Now, if both \(i\) or \(j\) are strictly bigger than 0, setting \(r = s = 0, k = 2\) implies \(a_{(-2)}b \in C_1(V)\). If only one is strictly bigger than 0, setting the respective index \(r\) or \(s\) to 1 and the other to 0 and \(k = 1\), the same conclusion holds. Thus, \(C_2(V) \subseteq C_1(V) + V_0\). Since \(V_0\) is finite-dimensional, \(C_2\)-cofiniteness implies the claim.
\end{proof}
In this article, we will be concerned exclusively with the cases \(n = 1, 2\). For an extended discussion of \(C_n\)-cofiniteness properties, see, for instance, \cite[Section 3]{Li}.

\subsection{The space of one-point functions}
Here, we recall the definition of the space of one-point functions. We will not work directly with this definition, but we include it for completeness.
In \cite[Theorem 4.2.1]{Z}, Zhu constructs for a given VOA \((V, Y, \mathbf{1}, \omega)\) an associated VOA \((V, Y[-, z], \widetilde{\omega}, \mathbf{1})\) with a transformed grading. We denote the corresponding transformed vertex operators as \(Y[v, z] = \sum_{m \in \mathbb{Z}}v[m]z^{-m-1}\) for \(v \in V\) and the Virasoro modes by \(Y[\widetilde{\omega}, z] = \sum_{m \in \mathbb{Z}}L_{[m]}z^{-m-2}\). 

Further, denote by \(\mathbb{H} = \{\tau \in \mathbb{C}\,|\, \Im(\tau) > 0\}\) the complex upper half plane and by \(E_{2k} \colon \mathbb{H} \rightarrow \mathbb{C}\) the normalized Eisenstein series for \(k \geq 1\). Recall that \(E_{2k}\) are modular forms only for \(k \geq 2\), and \(E_4\) and \(E_6\) generate the algebra of modular forms \(\C[E_4, E_6]\).

With these preliminaries, we can give the definition of the vector space of one-point functions (see, for instance, \cite[Definition 3.1]{AN}).
\begin{definition}
\label{def1P}
Let \(\mathcal{O}(V)\) be the \(\mathbb{C}[E_4, E_6]\)-submodule of \(V \otimes_{\mathbb{C}} \mathbb{C}[E_4, E_6]\) generated by the elements 
\begin{align*}
    u[0]v, ~~~ u[-2]v + \sum_{k=2}^{\infty}(2k - 1)u[2k-2]v \otimes E_{2k}.
\end{align*}
The \textit{space of one-point functions on the torus} \(\mathcal{C}_1(V)\) consists of maps \({S \colon V \otimes_{\mathbb{C}} \mathbb{C}[E_4, E_6] \times \mathbb{H}} \rightarrow \mathbb{C}\) satisfying the following list of properties:
\begin{enumerate}[(C1)]
    \item For any \(v \in V\), the function \(S(v, -)\) is holomorphic on the complex upper half plane \(\mathbb{H}\).
    \item \label{C2} \sloppy For any \(\tau \in \mathbb{H}\), modular forms \(f_1, \hdots, f_k \in  \mathbb{C}[E_4, E_6]\) and vectors \(v_1, \hdots, v_k \in V\), the condition \({S(\sum_{i = 1}^kv_i \otimes f_i, \tau) = \sum_{i = 1}^kf_i(\tau)S(v_i, \tau)}\) holds.
    \item \(S(x, \tau) = 0\) for every \(x \in \mathcal{O}(V)\).
    \item \label{C4} For any \(v \in V\), \(\tau \in \mathbb{H}\) and \(q = e^{2\pi i \tau}\), the map \(S\) satisfies the differential equation
    \begin{align*}
        S(L_{[-2]}v, \tau) = (2\pi i)^2q\partial_qS(v, \tau) + \sum_{k = 1}^{\infty}E_{2k}(\tau)S(L_{[2k-2]}v, \tau).
    \end{align*}
\end{enumerate}
\end{definition}
In \cite[Theorem 5.3]{DLM3} it was proved that the space of one-point functions \(\mathcal{C}_1(V)\) carries a natural action by the modular group \(\mathrm{SL}_2(\mathbb{Z})\).
\begin{example}
    \label{exComm}
    Let \(B\) be a finite-dimensional algebra and consider the VOA \(V^B\) with underlying set \(B\), with VOA-structure induced by setting \(\mathbf{1} = 1_B, Y(a, z)b = a \cdot b\) and \( \omega = 0\) (see \cite[Remark 3.4.5]{LL}). Since \(L_n = 0\) for all \(n \in \mathbb{N}\), one has \(V^B_0 = V^B\) and \(V_m^B = 0\) for \(m \neq 0\). Clearly, \(V^B\) is \(C_2\)-cofinite. Here, we have a natural identification \(\calc_1(V^B) \cong B^*\). This follows readily from the defining conditions:

    The assumptions imply for \(a \in V^B\) that \(a_{[m]} = a_{(m)}\), which is trivial except for \(m = -1\). Thus, \(\calo(V^B) = 0\). By \labelcref{C4}, we have further for any \(S \in \calc_1(V^B)\) that \(\partial_{\tau}S(a, \tau) = 0\), so \(S(a, \tau) = \Tilde{S}(a)\) for all \(\tau \in \mathbb{H}\) and a unique \(\Tilde{S} \in B^*\). Since \labelcref{C2} forces \(S(a \otimes f, \tau) = f(\tau)\tilde{S}(a)\) for all \(f \in \C[E_4, E_6]\), the map \(S \mapsto \tilde{S}\) gives an iso of \(\mathrm{SL}_2(\mathbb{Z})\)-modules between \(\calc_1(V^B)\) and \(B^*\), where \(B^*\) is equipped with the trivial action.
\end{example}

\subsection{Higher Zhu algebras and admissible \(V\)-modules}
We recall the definition of (higher) Zhu algebras, which were introduced in \cite{Z, DLM2}. They associate for every nonnegative integer an associative algebra \(A_n(V)\) to a VOA \(V\), capturing much of the representation-theoretic data. For every \(n \in \mathbb{N}\), let \(O_n(V)\) be the span of elements \(u \circ_n v\) and \(L_{-1}u + L_0u\) for homogeneous \(u, v \in V\), where \(\circ_n\) is defined by
\begin{align*}
    u \circ_n v = \Res_z Y(u, z)v\frac{(1+z)^{\wt u + n}}{z^{2n+2}},
\end{align*}
and set \(A_n(V) = V/O_n(V)\). Together with the multiplication on \(V\) defined by the formula
\begin{align*}
    u *_n v = \sum_{m = 0}^{n}(-1)^m\binom{m + n}{n}\Res_z Y(u, z)v\frac{(1+z)^{\wt u + n}}{z^{n + m + 1}},
\end{align*}
which descends onto the quotient, one has the following theorem.
\begin{theorem}[{\cite[Theorem 2.3]{DLM2}}]
    \(A_n(V)\) with the multiplication \(*_n\) is an associative algebra with unit \(\mathbf{1} + O_n(V)\). Moreover, \(\omega + O_n(V)\) is a central element.
\end{theorem}
We denote the equivalence class of an element \(a \in V\) in \(A_n(V)\) by \([a]\). As shown in \cite[Proposition 2.4]{DLM2}, there are natural algebra epimorphisms \(\pi_n \colon A_n(V) \rightarrow A_{n-1}(V)\) induced by the identity map on \(V\). 

An important result for the representation theory of \(A_n(V)\) is the following.
\begin{theorem}[{\cite[Theorem 2.5]{Miy}}]
\label{propFMSM}
Let \(V\) be a \(C_2\)-cofinite VOA. Then:
\begin{enumerate}[(i)]
    \item \(A_n(V)\) is a finite-dimensional algebra.
    \item There are only finitely many isomorphism classes of simple admissible \(V\)-modules and any simple admissible \(V\)-module is an ordinary module. 
\end{enumerate}
\end{theorem}

\begin{example}
     Consider again for \(B\) a commutative finite-dimensional algebra the VOA \(V^B\) from \cref{exComm}. It is easy to see that the \(n\)'th Zhu algebra, for every \(n \in \mathbb{N}\), coincides with the underlying algebra \(B\). In fact, one has \(L_{-1} + L_0 = 0\) and 
\begin{align*}
  a \circ_n b &= \Res_z \frac{(1 + z)^{\wt a + n}}{z^{2n+2}}Y(a, z)b = \Res_z \sum_{i = 0}^{\infty}\binom{n}{i}z^{i - 2n - 2}ab = \binom{n}{1 + 2n}ab = 0,  
\end{align*}
showing that \(O_n(V) = 0\). For the product \(*_n\), we compute
\begin{align*}
a *_n b &= \sum_{m = 0}^{n}(-1)^m\binom{m + n}{n}\Res_z Y(a, z)b\frac{(1+z)^{\wt a + n}}{z^{n + m + 1}} \\ &= \sum_{m = 0}^{n}(-1)^m\binom{m + n}{n}\Res_z \sum_{i = 0}^{\infty}\binom{n}{i}z^{i - n - m - 1}ab = \sum_{m = 0}^{n}(-1)^m\binom{m + n}{n}\binom{n}{m+n}ab \\ &= \binom{n}{n}^2ab = ab,
\end{align*}
and therefore, \((A_n(V^B), *_n) = (B, \cdot)\) as claimed.
\end{example}

One can also associate to \(V\) a Lie algebra \(\mathfrak{g}(V)\), obtained as a quotient of the space \(\C[t, t^{-1}] \otimes_{\C} V\) (see \cite[Section 3]{DLM2}). We denote by \(v(m)\) the equivalence class corresponding to \(t^m \otimes v\). Setting \(\wt (t^m \otimes v) = \wt v - m - 1\), there is an induced \(\mathbb{Z}\)-grading \(\mathfrak{g}(V) = \oplus_{m \in \mathbb{Z}}\mathfrak{g}(V)_m\). By \cite[Lemma 5.1]{DLM1}, every weak \(V\)-module \(M\) carries a natural action of \(\mathfrak{g}(V)\). One defines the vector space
\begin{align*}
    \Omega_n(M) = \{w \in M | \mathfrak{g}(V)_{-k} \triangleright w = 0 \; \mathrm{if}\; k > n\}.
\end{align*}
As shown in \cite[Theorem 3.2]{DLM2}, there is an action of algebras \(A_n(V) \rightarrow \End_{\mathbb{C}}\Omega_n(M), [a] \mapsto a(\mathrm{wt} a - 1) \triangleright \id_{\Omega_n(M)}\) and \(\Omega_n\) gives a functor from the category of weak \(V\)-modules to the category of \(A_n(V)\)-modules. One also has:
\begin{proposition}[{\cite[Proposition 3.4]{DLM2}}]
    \label{propOmega}
    Let \(M\) be an admissible \(V\)-module. Then \(\bigoplus_{i = 0}^n M(i) \subseteq \Omega_n(M)\), where equality holds if \(M\) is simple (assuming \(M(0) \neq 0\)). Moreover, in the latter case, for every \(0 \leq i \leq n\), \(M(i)\) is a simple \(A_n(V)\)-module.
\end{proposition}
Conversely, one can construct from an \(A_{n}(V)\)-module \(X\) two admissible \(V\)-modules \(\overline{M}_n(X)\) and \(L_n(X)\), where the latter is a quotient of the former. First, let \(\mathfrak{g}(V)_{-p} \triangleright X = 0\) for \(p > n\) and set \(P_n = \oplus_{p > n}\mathfrak{g}(V)_{-p} \oplus \mathfrak{g}(V)_{0}\). Then, by means of the universal enveloping algebra \(U(\mathfrak{g}(V))\), set
\begin{align*}
    M_n(X) = U(\mathfrak{g}(V)) \otimes_{U(P_n)}X.
\end{align*}
Assigning \(\mathrm{wt}(X) = n\) and \(M_n(X)(m) = U(\mathfrak{g}(V))_{m-n}X\) defines a \(\mathbb{Z}\)-grading on \(M_n(X)\). Note that for \(m < 0\), \(M_n(X)(m) = 0\). The operator defined by \(Y_{M_n(X)}(v, z) = \sum_{m \in \mathbb{Z}}v(m)z^{m-1}\) does not quite induce the structure of a \(V\)-module, as weak associativity is not satisfied. Dividing out the necessary relations, one obtains an admissible \(V\)-module \(\overline{M}_n(X)\). For details, we refer to \cite[Section 4]{DLM2}.

\begin{theorem}[{\cite[Theorem 4.1]{DLM2}}]
\label{thmMUniv}
\(\overline{M}_n(X)\) is an admissible \(V\)-module with grading \(\overline{M}_n(X) = \bigoplus_{m = 0}^{\infty}\overline{M}_n(X)(m)\) and \(\overline{M}_n(X)(n) = X\) with the following universal property: for any weak \(V\)-module \(M\), a morphism of \(A_n(V)\)-modules \(\phi \colon X \rightarrow \Omega_n(M)\) can be uniquely extended to a morphism \(\overline{\phi} \colon \overline{M}_n(X) \rightarrow M\) of weak \(V\)-modules.
\end{theorem}
We remark that \(\overline{M}_n(X)(0) \neq 0\) if and only if the \(A_n(V)\)-action on the module \(X\) does not factor through \(A_{n-1}(V)\).

The module \(L_n(X)\) is a certain quotient of the universal module \(\overline{M}_n(X)\). We do not recall the construction here, but state its most important properties.

\begin{theorem}[{\cite[Theorem 4.2]{DLM2}}]
    \label{thmLn}
     Let \(X\) be an \(A_n(V)\)-module. Then \(L_n(X)\) is an admissible \(V\)-module that satisfies \(L_n(X)(n) \cong X\) as \(A_n(V)\)-modules. Moreover, if \(X\) is simple, then also \(L_n(X)\) is. If the \(A_n(V)\)-action does not factor through \(A_{n-1}(V)\), then \(L_n(X)(0) \neq 0\). 
\end{theorem}
Later, we will need some information on the grading of \(A_n(V)\) by generalized eigenvalues.
\begin{proposition}
    \label{propGrad}
    For any \(n \in \mathbb{N}\), there is a subset \(\Lambda(n) \subset \C\) of generalized eigenvalues with respect to multiplication by the central element \([\omega] \in A_n(V)\) such that \(A_n(V) = \bigoplus_{r \in \Lambda(n)}A_n(V)_{(r)}\),
    where \(A_n(V)_{(r)}\) is the generalized eigenspace with generalized eigenvalue \(r\). Moreover, the set of conformal weights \(\Lambda\) is contained in \(\Lambda(n)\), and \(\Lambda(n) \subseteq \{r + i | r \in \Lambda, 0 \leq i \leq n\}\).
\end{proposition}
\begin{proof}
    Any finite-dimensional algebra admits a generalized eigenspace decomposition as above, where the generalized eigenvalues come from the action of the central element on simple \(V\)-modules, see for instance \cite[Proposition 2.4.1]{AN2}. Hence, what we need to prove is the second claim. Note first that surjectivity of the morphism of algebras \(\pi_n \colon A_n(V) \rightarrow A_{n-1}(V)\) implies that for \(n \geq k\), \(\Lambda(k) \subseteq \Lambda(n)\). Let \(M\) be a simple \(V\)-module with conformal weight \(r\). Then, by \cref{propOmega}, \(M_r = M(0) = \Omega_0(M)\) is a simple \(A_0(V)\)-module on which \([\omega]\) acts by multiplication with \(r\), and hence \(\Lambda \subset \Lambda(n)\). For second inclusion stated in the proposition, we prove that every simple \(A_n(V)\)-module \(W\) arises as a graded component \(M(i)\) of some a simple ordinary \(V\)-module \(M\) with \(M(0) \neq 0\), where \(0 \leq i \leq n\). If \(W\) is not an \(A_{n-1}(V)\)-module, then by \cref{thmLn}, \(L_n(W)\) is a simple \(V\)-module so that \(W \cong L_n(W)(n)\) and \(L_n(W)(0) \neq 0\). If \(M\) is an \(A_{n-1}(V)\)-module, and not an \(A_{n-2}(V)\)-module, then \(L_{n-1}(W)(n-1) \cong W\). Iterating and finally using that if \(W\) is an \(A_0(V)\)-module (without further condition), then \(L_0(W)(0) \cong W\), we obtain the claimed result. Using \([\omega] \triangleright M(i) = (r + i)M(i)\) for \(M(0) \neq 0\) yields the second claimed inclusion.
\end{proof}

\section{Pseudo-Traces and a generating set for \(\mathcal{C}_1(V)\)}
\label{secPseudo}
\subsection{Pseudo-Trace maps}
\label{subsecPseudo}
In this section, we review the concept of pseudo-trace maps. For a finite-dimensional algebra \(A\) and a projective right\footnote{Of course, pseudo-traces can be defined equally well for left \(A\)-modules.} \(A\)-module \(W_A\), these are functionals on \(\End_A(W)\) defined from the data of symmetric functions on \(A\). Similarly to traces, they satisfy a cyclicity property.

Recall the following basic fact: A finitely generated right \(A\)-module \(W\) is projective if and only if
    \begin{align*}
    T\colon W \otimes_A \Hom_A(W, A) \rightarrow \End_A(W), \,\,\, w \otimes \varphi \mapsto w \triangleleft \varphi(-)
    \end{align*}
is an epimorphism, which is the case if and only if \(T\) is an isomorphism (see \cite[Theorem 1.7]{Br}).
This, in turn, is equivalent to the existence of subsets \(\{u_i\}_{i = 1}^m \subset W\) and \(\{\alpha_i\}_{i = 1}^m \subset \Hom_A(W, A)\) satisfying
\begin{align}
\label{eqPCS}
     w = \sum_{i = 1}^mu_i \triangleleft\alpha_i(w)
\end{align}
for all \(w \in W\).
To see the equivalence, note that \cref{eqPCS} implies that \(\sum_{i = 1}^m f(u_i) \otimes \alpha_i\) is a preimage of \(f \in \End_A(W)\) under \(T\). 
\begin{definition}
    A collection \(\{u_i, \alpha_i\}_{i = 1}^m\) satisfying \cref{eqPCS} is called an \(A\)-\textit{coordinate system} of \(W\).
\end{definition}
Let us further introduce the \(\C\)-vector space of \textit{symmetric linear functions on} \(A\), 
\begin{align*}
    \mathcal{S}(A) = \{\phi \in \Hom_{\mathbb{C}}(A, \mathbb{C}) | \phi(ab) = \phi(ba)\}.
\end{align*}
Note the isomorphism \(\mathcal{S}(A) \cong (A/[A,A])^*\), where \([A,A] = \spann_{\mathbb{C}}\{[a,b] |a, b \in A\}\) is the linear subspace spanned by commutators in \(A\).
\begin{definition}
A pair \((A, \phi)\) of an algebra \(A\) together with a map \(\phi \in \mathcal{S}(A)\) is called \textit{symmetric algebra} if \(\phi\) induces a non-degenerate pairing \(\langle -, -\rangle : A \times A \rightarrow \mathbb{C}\), that is \(\langle a,b \rangle = \phi(ab)\) for \(a, b \in A\).
\end{definition}
The notion of symmetric functions is important to define the concept of pseudo-trace maps.
\begin{definition}[{\cite[Section 2]{A}}]
    \label{defPseudo}
    For a finite-dimensional algebra \(A\), a finitely generated module \(W_A\), and a symmetric function \(\phi \in \mathcal{S}(A)\), the linear map \(\tr_W^{\phi} \in \End_A(W)^*\) defined by
\begin{align*}
    \tr_W^{\phi}(f) = \sum_{i = 1}^m\phi(\alpha_i(f(u_i)))
\end{align*}
for \(f \in \End_A(W)\) is called a \textit{pseudo-trace map}. Similarly, we define pseudo-traces for left \(A\)-modules.
\end{definition}
The definition is independent of the choice of \(A\)-coordinate system: Let \(S \colon W \otimes_A \Hom_A(W, A) \rightarrow A/[A, A], \varphi \otimes w \mapsto \varphi(w) + [A, A]\) and denote for \(\phi \in \mathcal{S}(A)\) by \(\Bar{\phi} \in (A/[A, A])^*\) the induced map on the quotient. A short computation shows that \(\tr_W^{\phi} = \Bar{\phi} \circ S \circ T^{-1}\), where the right hand side is defined independently of any \(A\)-coordinate system.

The most important property of pseudo-trace maps is that they satisfy a cyclicity property (see \cite[Proposition 4.2]{AN}): If \(W\) and \(X\) are finitely generated projective right \(A\)-modules,
\begin{align*}
    \tr_W^{\phi}(f \circ g) = \tr_X^{\phi}(g \circ f) \, \, \text{ for any } \, \, g \in \Hom_A(W, X) \text{ and } f \in \Hom_A(X, W).
\end{align*}
Let us look at an example of a symmetric algebra and a corresponding pseudo-trace map.
\begin{example}
    Consider the matrix algebra \(A = \left\{Y =\begin{pmatrix} y_1 & y_2 \\ 0 & y_1 \end{pmatrix} \, | \, y_1, y_2 \in \mathbb{C} \right\} \subset \mathrm{Mat}_2(\mathbb{C})\) with its canonical basis \(\Omega = \{e, f\}\) given by the identity \(e = \mathbf{1}_2\) and \(f = \begin{pmatrix} 0 & 1 \\ 0 & 0 \end{pmatrix}\). The map \(\phi \colon A \rightarrow \mathbb{C}, \phi(Y) = y_2\) is a symmetric function on \(A\) that induces a non-degenerate symmetric form given by
    \begin{align*}
        \langle -, - \rangle: A \times A \rightarrow \mathbb{C}, \, \, \, \langle Y, Z \rangle = \phi(YZ) = y_1z_2 + y_2z_1,
    \end{align*}
    turning the pair \((A, \phi)\) into a symmetric algebra. The vector space \(W = W_1 \oplus W_2\) with \(W_i = \C^m\) becomes a right \(A\)-module via matrix multiplication on vectors \((u,v) \in W_1 \oplus W_2\) from the right. The \(A\)-module \(W\) has a basis \(\{v_j^e \triangleleft e, v_j^e \triangleleft f\}_{j=1}^m\), where \(v_j^e = ((0, \hdots, 0, 1, 0, \hdots, 0), 0)\), with 1 in the \(j\)'th entry. Note that the vectors \(v_j^e\) and the \(A\)-linear maps \(\beta_j \colon W \rightarrow A, \beta_j((u,v)) = \begin{pmatrix} u_j & v_j \\ 0 & u_j \end{pmatrix}\) provide an \(A\)-coordinate system for \(W\), so \(W\) is indeed a projective right \(A\)-module. Finally, for \(g \in \End_AW\), we may write \(g(v_j^e) = \sum_{t = 1}^m\sum_{\rho \in \Omega}\alpha_{jt}^{\rho}v_t^{e} \triangleleft \rho\) and find for the pseudo-trace evaluated at \(g\)
    \begin{align*}
    \tr_W^{\phi}(g) = \phi\left(\sum_{j=1}^m\beta_j(g(v_j^e))\right) = \sum_{j,t=1}^m\sum_{\rho \in \Omega}\alpha_{jt}^{\rho}\phi(\beta_j(v_t^{e} \triangleleft \rho)) = \sum_{j,t=1}^m(\alpha_{jt}^e\phi(\delta_{jt}e) + \alpha_{jt}^f\phi(\delta_{jt}f)) = \sum_{j=1}^m\alpha_{jj}^f,
\end{align*}
that is, expressing \(g\) as the matrix \(\begin{pmatrix} \alpha^e & \alpha^f\\ 0 & \alpha^e \end{pmatrix}\), where \(\alpha^e, \alpha^f \in \mathrm{Mat}_m(\mathbb{C})\), we find \(\tr_W^{\phi}(g) = \tr(\alpha^f)\).
\end{example}
The following serves as an alternative viewpoint on pseudo-traces, presented in \cite[Section 4]{GR}. We will use this description later in \cref{secCat}. To accommodate to the situation there, we switch conventions and review the concept in terms of left \(A\)-modules. Let \(_AW\) be a finitely generated projective \(A\)-module. Let \(X\) be a finite-dimensional vector space with a surjective \(A\)-module map \(\pi\colon A \otimes_\mathbb{C} X \rightarrow W\). Due to projectivity, there is an injective \(A\)-morphism \(\iota\colon W \rightarrow A \otimes_\mathbb{C} X\) satisfying \(\pi \circ \iota = \id_W\). 

Choose a basis \(\{x_1, \hdots, x_m\}\) of \(X\) and the corresponding dual basis \(\{x_1^*, \hdots, x_m^*\}\) of \(X^*\). Consider the standard evaluation \(\widetilde{\mathrm{ev}}_X\colon X  \otimes_{\mathbb{C}} X^* \rightarrow \mathbb{C}, x \otimes \varphi \mapsto \varphi(x)\) and coevaluation \(\mathrm{coev}_X\colon \mathbb{C} \rightarrow X  \otimes_{\mathbb{C}} X^*, \lambda \mapsto \lambda \sum_{i=1}^mx_i \otimes x_i^*\) that form (part of) the rigid structure of \(\vect_{\C}^{\mathrm{fd}}\). Consider for \(f \in \End_A(W)\) the map
\begin{align*}
    \Psi_f \colon A \cong A \otimes_{\mathbb{C}} \mathbb{C} \xrightarrow{\id_A \otimes \, \mathrm{coev}_X} A \otimes_{\mathbb{C}} X  \otimes_{\mathbb{C}} X^* \xrightarrow{\pi \, \otimes \, \id_{X^*}} W  \otimes_{\mathbb{C}} X^* \xrightarrow{f \, \otimes \, \id_{X^*}} \\ W  \otimes_{\mathbb{C}} X^* \xrightarrow{\iota \, \otimes \, \id_{X^*}} A  \otimes_{\mathbb{C}} X  \otimes_{\mathbb{C}} X^* \xrightarrow{\id_A \otimes \, \widetilde{\mathrm{ev}}_X} A  \otimes_{\mathbb{C}} \mathbb{C} \cong A \xrightarrow{\phi} \mathbb{C},
\end{align*}
with the canonical identification \(A  \otimes_{\mathbb{C}} \mathbb{C} \cong A\). We represent morphisms graphically as, for example
\begin{align*}
\\[-2em]
    f \,\, = \, \, \adjincludegraphics[valign=c, scale = 0.9]{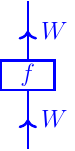} \, \, \, \, \, \, \, \, \, \, \, \, \, \, \, \,
    \widetilde{\mathrm{ev}}_X \,\, = \, \, \adjincludegraphics[valign=c, scale = 0.9]{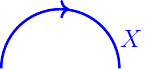}\, \, \, \, \, \, \, \, \, \, \, \, \, \, \, \,
    \mathrm{coev}_X \,\, = \, \, \adjincludegraphics[valign=c, scale = 0.9]{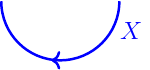} 
    \\[-2em]
\end{align*}
using the standard graphical calculus for rigid monoidal categories, here for the category \(\vect_{\C}^{\mathrm{fd}}\), where diagrams are read from bottom to top, tensor products are represented by horizontal juxtaposition, and the strands decorated by the monoidal unit are not drawn. We will not go into detail here about the rules of the graphical calculus, since we are only going to make use of it in one instance, namely in the proof of \cref{propPGPT} in \cref{secCat}. The interested reader may consult sources such as \cite{TV} for details. We can then represent the map \(\Psi_f\) pictorically through the image below: 
\begin{align*}
    \Psi_f \, \, = \, \,  \adjincludegraphics[valign=c, scale = 0.9]{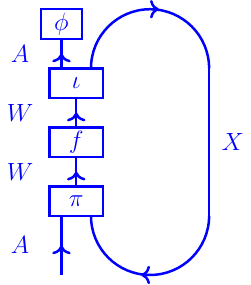}.
\end{align*}
As our goal is to encode pseudo-traces, we need to extract them from the linear map \(\Psi_f\).

\begin{lemma}
    \label{lemPTCM}
    For \(\phi \in \mathcal{S}(A)\) and \(f \in \End_A(W)\), the scalar \(\Psi_f(1_A)\) is independent of the choices of \(X, \iota\) and \(\pi\), and the pseudo-trace map \(\tr_{W}^{\phi} \colon \End_A(W) \rightarrow \mathbb{C}\) is obtained by \(\tr_{W}^{\phi}(f) = \Psi_f(1_A)\).
\end{lemma}
\begin{proof}
We want to obtain an \(A\)-coordinate system from the given data. Thus, set \(u_i = \pi(1_A \otimes x_i)\) and \(\alpha_i = \mu \circ (\id_A \otimes \, x_i^*) \circ \iota\) for \(1 \leq i \leq m\). This can be cast in the equivalent form \(\iota(w) = \sum_{j=1}^m\alpha_j(w) \otimes x_j\) for any \(w \in W\). First, let us check that this defines an \(A\)-coordinate system. For \(w \in W\),
\begin{align*}
    \sum_{i=1}^m\alpha_i(w) \triangleright u_i = \pi\Big(\sum_{i=1}^m\alpha_i(w) \otimes x_i\Big) = (\pi \circ \iota)(w) = w.
\end{align*}
Second, let us explicitly compute \(\Psi_f(1_A)\) and compare with \(\tr_{W}^{\phi}(f)\). We have
\begin{align*}
   \Psi_f(1_A) &= \phi \circ (\id_A \otimes \, \widetilde{\mathrm{ev}}_X) \circ ((\iota \circ f \circ \pi) \otimes \id_{X^*}) \Big( \sum_{i = 1}^m 1_A \otimes x_i \otimes x_i^* \Big) \\
   &= \phi \Big(\sum_{i = 1}^m(\id_A \otimes \, \widetilde{\mathrm{ev}}_X)(\iota(f(u_i)) \otimes x_i^*) \Big) 
    = \phi \Big(\sum_{i = 1}^m(\id_A \otimes \, \widetilde{\mathrm{ev}}_X)\Big(\sum_{j = 1}^m\alpha_j(f(u_i)) \otimes x_j \otimes x_i^* \Big) \Big) \\
    &= \phi \Big(\sum_{i,j = 1}^m \alpha_j(f(u_i))\delta_{ij} \Big) = \phi \Big(\sum_{i= 1}^m \alpha_i(f(u_i)) \Big) = \tr_{W}^{\phi}(f),
\end{align*}
as desired. By this argument, it is also clear that \(\Psi_f(1_A)\) is independent of the choice of \(X, \pi\) and \(\iota\), because \(\tr_{W}^{\phi}\) is independent of the choice of \(A\)-coordinate system. 
\end{proof}

Decompose the unit of \(A\) as \(1_A = \sum_{i=1}^k\sum_{j = 1}^{k_i}e_{ij}\) where \(e_{ij}\) are mutually orthogonal primitive idempotents such that \(e_{ij}A \cong e_{i\ell}A\) and \(e_{ij}A \ncong e_{rs}A\) for \(i \neq r\). There is an induced equivalence relation \(\sim\) on the set of primitive idempotents. Choose representatives \(e_i = e_{i1}\) of the equivalence classes and set \(e = \sum_{i=1}^ke_i\). It is a common fact that \(eAe\), the \textit{basic algebra associated to} \(A\), is an algebra with unit \(e\) which is basic (\cite[Corollary 6.10]{L}). We can regard \(Ae\) as an \((A, eAe)\)-bimodule that is projective as both a left \(A\)- and right \(eAe\)-module.

Moreover, by \cite[Theorem 10.1]{NS}, \((A, \phi)\) is a symmetric algebra if and only if \((eAe, \phi|_{eAe})\) is. For a symmetric map \(\phi \in \mathcal{S}(A)\), the tuple \((A, \phi)\) does not necessarily form a symmetric algebra. The subspace \(\mathfrak{R}_{\phi} = \{a \in A| \phi(aA) = 0\}\) is a two-sided ideal in \(A\) such that \(\phi|_{\mathfrak{R}_{\phi}} = 0\). Thus, we may transport \(\phi\) to a well-defined symmetric function on the quotient \(A/\mathfrak{R}_{\phi}\), where we denote it by the same letter by abuse of notation. Then \((A/\mathfrak{R}_{\phi}, \phi)\) is a symmetric algebra.

\subsection{Miyamoto's pseudo-trace functions and modular invariance}
\label{subSecMiy}
In the following, we recall the necessary ingredients to state the main result of \cite{Miy}. We set the following convention for the rest of the article.
\begin{convention}
    \label{conv}
    \(V= \oplus_{n \in \mathbb{N}}V_n\) is a \(C_2\)-cofinite vertex operator algebra such that the Virasoro algebra acts nontrivially on every simple \(V\)-module.
\end{convention}
\(C_2\)-cofiniteness implies that \(V\)-modules have more structure:

\begin{theorem} [{\cite[Proposition 2.1]{AN}}, {\cite[Theorem 2.7]{Miy}}]
\label{thmWisA}
For \(V\) as above and \(M\) a non-zero finitely generated weak \(V\)-module, there are \(r_1, \hdots, r_k \in \mathbb{C}\) such that \(r_i - r_j \notin \mathbb{Z}\) for \(i \neq j\) and \(d_1, \hdots, d_k \in \mathbb{N}\setminus \{0\}\) such that \(M = \bigoplus_{i = 1}^k \bigoplus_{m = 0}^{\infty}M_{(r_i + m)}\), where \(M_{(r_i + m)} = \{w \in M|(L_0 - r_i - n)^{d_i }w = 0 \}\) is finite-dimensional and  \(M_{(r_i)} \neq 0\). In particular, \(M\) is a logarithmic module. 
\end{theorem}
Additionally, one can show that the numbers \(r_1, \hdots, r_k\) from \cref{thmWisA} are all conformal weights of \(V\), and that the set of conformal weights \(\Lambda\) is finite. Since every simple \(V\)-module is finitely generated, a consequence of the above theorem is that all simple weak \(V\)-modules are ordinary modules.

\smallskip
We review the concept of pseudo-trace functions (on the VOA-module level) as encountered in \cite{Miy}. The goal in \cite{Miy} is to establish a spanning set of the space of one-point functions \(\mathcal{C}_1(V)\) via symmetric functions on a higher Zhu algebra \(A_n(V)\), where \(n\) is a sufficiently large integer: By \cite[Lemma 3.5.1, Proposition 4.5.3]{AN2}\footnote{The preprint \cite{AN2} was subsequently withdrawn; however, version 1, which contains detailed proofs of several results of \cite{Miy}, remains available on arXiv.}, for any weak \(V\)-module \(M\), there is a natural number \(n_a \in \mathbb{N}\) such that the \(n\)'th graded component \(M(n) = \bigoplus_{i = 1}^kM_{(r_i + n)}\) is nonzero for every \(n > n_a\). Miyamoto then uses the modules \(L_n(T)\) for suitable \(A_n(V)\)-modules \(T\) and defines pseudo-trace functions on them to obtain his spanning set (see \cref{thmMiy} below). 

To state the result, we first require a bit of preparation. Recall that for a \(C_2\)-cofinite vertex algebra, \(A_n(V)\) is of finite dimension by \cref{propFMSM}. Consider the algebra \(A_n(V)/\mathfrak{R}_{\phi}\), where \(\phi \in \mathcal{S}(A_n(V))\) is some symmetric function. Denote by \(e\), as in \cref{subsecPseudo}, the sum of representatives of equivalence classes of idempotents with respect to the equivalence relation \(\sim\). We can form the associated basic algebra \(B = e\left(A_n(V)/\mathfrak{R}_{\phi} \right)e \). Furthermore, let \(T = \left(A_n(V)/\mathfrak{R}_{\phi} \right)e\), which naturally carries the structure of an \((A_n(V), B)\)-bimodule.

\looseness = 1 By \cite[Lemma 3.5.1]{AN2}, there is a nonnegative integer \(n_b\) such that for \(n > n_b\), under the condition that the Virasoro algebra acts nontrivially on every simple \(V\)-module \(M\), the components \(M_{r+n}\) are nonzero. Here, \(r\) is the conformal weight of \(M\). Let further \(n_a\) be an even larger integer with \(n_a > n_b + \mathrm{max}_{r, s \in \Lambda}\{|\Re(r) - \Re(s)|\}\) (for an extended discussion of this choice, see \cite[Definition~4.5.2, Proposition~4.5.3, Section 8.1]{AN2}).

\begin{theorem}[{\cite[Lemma 4.5.1, Theorem 4.5.5]{AN2}}]
\label{thmLproj}
The \(V\)-module \(L_n(T)\) is a right \(B\)-module with \(B\)-action defined by \(\left( v_1(i_1) \hdots v_k(i_k)t\right) \triangleleft b = v_1(i_1) \hdots v_k(i_k)(tb)\) for every \(t \in T, b \in B\). 
For \(n > n_a\), every component \(L_n(T)(m)\) is a finitely generated projective right \(B\)-module.
\end{theorem}
Note that \(L_n(T)\) is, in fact, a \((V, B)\)-bimodule. In the proof of Miyamoto's main result, he uses symmetric functions \(\phi\) that satisfy for any \([a] \in A_n(V)\) that \(\phi(([\omega] - c_V/24 - n - r)^{d+1} *_n [a]) = 0\) for suitable \(d \in \mathbb{N}\) and \(r \in \mathbb{C}\). Adding this assumption on \(\phi\), by definition of \(T\), we have 
\begin{align*}
    ([\omega] - c_V/24 - n - r)^{d+1}T = 0.
\end{align*}
\looseness=1 This implies that for any \(x = v_1(i_1) \hdots v_k(i_k)t \in L_n(T)(m)\) that \(v_1(i_1) \hdots v_k(i_k)(([\omega] - n) *_n t) = (L_0 - m)x\), using standard commutator identities for \(L_0\). Accordingly, on every component of \(L_n(T)\), 
\begin{align}
    \label{eqNP}
    (L_0 - m - r - c_V/24)^{d+1}L_n(T)(m) = 0,
\end{align}
\looseness = -1 which says that any component is a generalized eigenspace for the operator \(L_0\). Expressed in symbols, this means \(L_n(T)(m) = L_n(T)_{(r + c_V/24 + m)}\). Decompose the operator \(L_0 = L_0^s + (L_0 - L_0^s)\), where
\begin{align*}
    L_0^s|_{L_n(T)(m)} = (m + r + c_V/24)\id_{L_n(T)(m)},
\end{align*}
into its semisimple and nilpotent part. If \(V\) were rational, \(L_0\) would act semisimply on \(L_n(T)\), and we could define for \(\tau \in \mathbb{H}\) the operator \(q^{L_0} = q^{r+m+c_V/24}\) on any component subspace \(L_n(T)(m)\). But if \(V\) is not rational as in the present case, the best we can achieve are generalized eigenspace components. By using the decomposition of \(L_0\) and \cref{eqNP}, we can make sense of the expression \(q^{L_0}\) as
\begin{align*}
    q^{L_0-c_V/24} = q^{L_0^s - c_V/24 + (L_0 - L_0^s)} = q^{m+r}\sum_{j=0}^d\frac{1}{j!}(2 \pi i \tau (L_0 - c_V/24 - m - r))^j.
\end{align*}
For the following definition, note that \(o(v), q^{L_0}\) are grade-preserving operators for any \(v \in V\), and also \(B\)-endomorphisms -- that is, \(o(v), q^{L_0} \in \End_B(L_n(T)(m))\) for any \(m \in \mathbb{N}\). 
\begin{proposition}[{\cite[Section 4.1, Corollary 5.4]{Miy}}]
    The function \(\xi_{L_n(T)}^{\phi, B} \colon V \otimes_{\C} \C[E_4, E_6] \times \mathbb{H} \rightarrow \mathbb{C}\),
    \begin{align*}
        \xi_{L_n(T)}^{\phi, B}(v \otimes f, \tau) = f(\tau)\tr_{L_n(T)}^{\phi}\left( o(v)q^{L_0 - c_V/24} \right) 
        = f(\tau)\sum_{m=0}^{\infty} \tr_{L_n(T)(m)}^{\phi}\left( o(v)q^{L_0 - c_V/24} \right),
    \end{align*}
called a \emph{pseudo-trace function}, is well-defined and a one-point function, that is \(\xi_{L_n(T)}^{\phi, B} \in \mathcal{C}_1(V)\).
\end{proposition}
\begin{remark}
    Without entering into the details, we mention that in \cite{Miy}, the underlying pseudo-trace maps are defined somewhat differently than in \cref{defPseudo}. However, the two approaches are reconciled in \cite{A}, and we can therefore freely make use of the latter description.
\end{remark}
We state the main result of \cite{Miy} alluded to earlier. 
\begin{theorem}[{\cite[Theorem 5.5]{Miy}}]
\label{thmMiy}
Let \(V = \bigoplus_{m = 0}^{\infty}V_m\) be a \(C_2\)-cofinite VOA such that the Virasoro algebra acts nontrivially on every simple \(V\)-module. Then, \(\mathcal{C}_1(V)\) is spanned by pseudo-trace functions. More precisely, for \(n > n_a\), there are symmetric linear functions \(\phi_1, \hdots, \phi_k \in \mathcal{S}(A_n(V))\) so that \((B_i, \phi_i)\), where \(B_i = e(A_n(V)/\mathfrak{R}_i)e\), are symmetric algebras, and \((A_n(V), B_i)\)-modules \(T_i = (A_n(V)/\mathfrak{R}_i)e\) with \(\mathfrak{R}_i = \mathfrak{R}_{\phi_i}\) such that the collection \(\left\{\xi_{L_n(T_i)}^{\phi_i, B_i}\right\}_{i = 1}^k\) forms a basis for \(\mathcal{C}_1(V)\). Moreover, 
\begin{align*}
    \dim_{\C} \calc_1(V) = \dim_{\C}A_n(V)/[A_n(V), A_n(V)] - \dim_{\C}A_{n-1}(V)/[A_{n-1}(V), A_{n-1}(V)].
\end{align*}
\end{theorem}
By \cite[Theorem 8.2.1]{AN2}, which provides additional detail for Miyamoto's result, we can assume the left \(A_n(V)\)-modules \(T_i\) to be such that the defining action does not factor through \(A_{n-1}(V)\).
\begin{remark}
    \label{remZhu}
    If \(V\) is in addition rational, then trace functions are sufficient to span \(\calc_1(V)\). More explicitly, if \(M^1, \hdots, M^{\ell}\) is a complete set of representatives of simple \(V\)-modules, then setting 
    \begin{align*}
        T_{M^i}(v, \tau) = \sum_{m=0}^{\infty} \tr_{M^i}\left( o(v)q^{L_0 - c_V/24} \right)
    \end{align*}
    using the ordinary trace, \cite[Theorem 5.3.1]{Z} states that \(\{T_{M^i}\}_{i = 1}^{\ell}\) is a basis for \(\calc_1(V)\).
\end{remark}
\begin{example}
    \label{exTr}
    Consider again the VOA \(V^B\) as defined in \cref{exComm}. One can lift any \(W \in B\lmod\) to a weak \(V^B\)-module \(\widehat{W}\) with the same underlying vector space and \(Y_{\widehat{W}}(a, z)b = a \triangleright b\). Conversely, any weak \(V^B\)-module is of the form \(\widehat{W}\) for a unique \(W \in B\lmod\). This follows since \(0 = Y_{\widehat{W}}(L_{-1}a, z) = \partial_zY_{\widehat{W}}(a, z)\) implies \(Y_{\widehat{W}}(a, z) = a_{(-1)}\), and weak associativity for \(Y_{\widehat{W}}\) implies that \(a \triangleright w := a_{(-1)}w \) defines a \(B\)-action on \(W\). Since \(L_0^{\widehat{W}} = 0\), we see that weak \(V^B\)-modules are in bijection with (possibly infinite-dimensional) \(B\)-modules, and ordinary \(V^B\)-modules are in bijection with finite-dimensional \(B\)-modules. It also follows, for \(\chi_W \colon B \rightarrow \C\) the character of \(W\), that \(T_W(a, \tau) = \chi_W(a)\) for \(a \in V^B, \tau \in \mathbb{H}\). 
    
    Now, let \(W^1, \hdots, W^{\ell}\) denote a complete set of representatives of simple \(B\)-modules. If \(B\) is semisimple, then a standard result in the representation theory of commutative algebras asserts that the collection of characters \(\chi_{W^1}, \hdots, \chi_{W^{\ell}}\) is a basis for the dual space \(B^{*}\). Since for \(B\) semisimple \(V^B\) is rational and by \cref{exComm}, \(\calc_1(V^B) \cong B^*\), that statement is equivalent to Zhu's theorem (cf. \cref{remZhu}), asserting that \(\{T_{\widehat{W^i}}\}_i\) is a basis for \(\calc_1(V^B)\). The situation changes drastically if \(B\) is not semisimple: For example, for the nonsemisimple two-dimensional algebra \(B = \C[\varepsilon]/(\varepsilon^2)\), there is only one irreducible representation up to isomorphism, on which \(t\) acts by 0. Thus, the two-dimensional space \(\calc_1(V^B)\) is not spanned by trace functions indexed by simple modules, so the assertion of Zhu's theorem does not hold.
    \end{example}

\subsection{A spanning set
of one-point functions in the sense of Arike--Nagatomo}
In this section, we reformulate the main result of \cite{Miy} in terms of pseudo-trace functions in the sense of \cite{AN}, which we first have to recall.

Let \(M\) be a finitely generated logarithmic \(V\)-module, which we assume in the form
\begin{align*}
    M = \bigoplus_{m = 0}^{\infty}M_{(r + m)} 
\end{align*}
(cf. \cref{thmWisA}), with \(r \in \mathbb{C}\) and each \(M_{(r + m)}\) being a generalized eigenspace of \(L_0\), that is, there exists \(d \in \mathbb{N}\setminus\{0\}\) such that \((L_0 - r - m)^{d}M_{(r + m)} = 0\). 
The vector space of \(V\)-endomorphisms \(\End_V(M)\) is an algebra for which \(M\) naturally forms a left module. Since any \(V\)-endomorphism is determined by its values on the finitely many generators of \(M\), the algebra \(\End_V(M)\) is finite-dimensional. Moreover, since \(V\)-homomorphisms are grade-preserving, any component \(M_{(r + m)}\) is a finite-dimensional \(\End_V(M)\)-module. Let \(P \subseteq \End_V(M)\) be a subalgebra such that each component \(M_{(r + m)}\) is projective as a \(P\)-module.\footnote{Equivalently, \(M\) is a projective \(P\)-module. Such an algebra always exists, for example \(P = \mathbb{C}\id_M\).} Summarizing, we have:
\begin{proposition}[{\cite[Proposition 4.3]{AN}}]
    \label{propP}
    Let \(V= \bigoplus_{m = 0}^{\infty}V_m\) a \(C_2\)-cofinite VOA and \(M\) a finitely generated logarithmic \(V\)-module. For an algebra \(P \subseteq \End_V(M)\) as above, \(P\) is finite-dimensional and each generalized eigenspace \(M_{(r + m)}\) of \(M\) is a finitely generated projective \(P\)-module.
\end{proposition}
Just as in the last section, we have a well-defined expression for \(q^{L_0}\). Let \(\phi \in \mathcal{S}(P)\) be a symmetric function. Due to \cref{propP}, we can define pseudo-trace maps \(\tr_{M_{(r + m)}}^{\phi} \colon M_{(r + m)} \rightarrow \mathbb{C}\) and thus obtain a pseudo-trace function in the same manner as in the last section, namely by setting 
\begin{align*}
    \xi_{M_{(r + m)}}^{\phi, P} \colon V \otimes_{\C} \C[E_4, E_6] \times \mathbb{H} \rightarrow \C, \, \, \, \xi_{M_{(r + m)}}^{\phi, P}(v \otimes f, \tau) = f(\tau)\tr_{M_{(r + m)}}^{\phi}\left(o(v)q^{L_0-c_V/24}\right)
\end{align*}
\looseness = -1 and further \( \xi_{M}^{\phi, P} = \sum_{m = 0}^{\infty} \xi_{M_{(r + m)}}^{\phi, P}\). For a finitely generated \(V\)-module with several lowest weights (but of course finitely many), we extend this assignment linearly. Since the substantial difference to Miyamotos pseudo-traces is only the algebra that is being used and not the construction of the pseudo-trace function itself, we keep the notation from earlier. Arike and Nagatomo prove the following result.
\begin{theorem}[{\cite[Theorem 4.1]{AN}}]
Let \(V = \bigoplus_{m = 0}^{\infty}V_m\) be a \(C_2\)-cofinite VOA and \(M\) a finitely generated logarithmic \(V\)-module. For a subalgebra \(P \subseteq \End_V(M)\) such that \(M\) is a projective left \(P\)-module and a symmetric linear function \(\phi\) on \(P\), the pseudo-trace function \(\xi_{M}^{\phi, P}\) is a well-defined one-point function, that is \( \xi_{M}^{\phi, P} \in \mathcal{C}_1(V)\).
\end{theorem}

Using the above notion of pseudo-trace function, we can rewrite \cref{thmMiy} in terms of \cite{AN}. 
\begin{proposition}
\label{thmPmain}
Let \(V = \bigoplus_{m=0}^{\infty}V_m\) be a \(C_2\)-cofinite VOA such that any simple \(V\)-module has a nontrivial action of the Virasoro algebra. For every \(n > n_a\), there are finite-dimensional symmetric algebras \((P_1, \Bar{\phi}_1) \hdots, (P_k, \Bar{\phi}_k)\) and \(P_i\)-modules \(T_i\), for \(i \in \{1, \hdots ,k\}\), such that the following are true:
\begin{itemize}
    \item For all \(i \in \{1, \hdots, k\}\), \(P_i\) is a subalgebra of \(\End_VL_n(T_i)\) and \(L_n(T_i)\) is a projective \(P_i\)-module.
    \item The space of one-point functions \(\calc_1(V)\) is spanned by the pseudo-trace functions \(\xi_{L_n(T_i)}^{\Bar{\phi}_i, P_i}\).
\end{itemize}
\end{proposition}
\begin{proof}
Due to \cref{thmMiy} there are symmetric functions \(\phi_1, \hdots, \phi_k\) on \(A_n(V)\) defining symmetric algebras \((B_i, \phi_i)\) and corresponding \((A_n(V), B_i)\)-bimodules \(T_i\). They are given by \(B_i = e(A_n(V)/\mathfrak{R}_i)e\) and \(T_i = (A_n(V)/\mathfrak{R}_i)e\), respectively, where \(\mathfrak{R}_i = \mathfrak{R}_{\phi_i}\). Recall that every \(L_n(T_i)\) is a \((V, B_i)\)-bimodule, and projective as a right \(B_i\)-module (cf. \cref{thmLproj}). Because of this, the action maps
\begin{align*}
    \rho_i \colon B_i^{\mathrm{op}} \rightarrow \End_{\mathbb{C}}L_n(T_i), \,\,\, \rho_i(b)x = x \triangleleft b \,\, \text{ for } b \in B_i, x \in L_n(T_i),
\end{align*}
take values in \(\End_VL_n(T_i)\), and we can set \(P_i = \rho_i(B_i^{\mathrm{op}})\). We now show that, in fact, \(P_i \cong B_i^{\mathrm{op}}\) as algebras. First, note that, by definition, every \(T_i\) is a faithful \(B_i\)-module: If for \(b = eae \in B_i\), \(t \triangleleft b = 0\) for all \(t \in T_i\), then in particular \(b = e^2ae = e \triangleleft b = 0\). 
We have \(L_n(T_i)(n) \cong T_i\) by \cref{thmLn}, so the right \(B_i\)-modules \(L_n(T_i)\) are faithful as well. Thus, the restrictions of the maps \(\rho_i\) to their images are isomorphisms of algebras, and the \(L_n(T_i)\) are projective as \(P_i\)-modules.

As a consequence, there are symmetric linear functions \(\Bar{\phi}_i \in \mathcal{S}(P_i)\) with \(\Bar{\phi}_i \circ \rho_i = \phi_i\) for all \(1 \leq i \leq k\). Since for any \(B_i\)-coordinate system \(\{\alpha_j, u_j\}_{j=1}^s\) for \(L_n(T_i)(m)\), where \(\alpha_j \in \Hom_{B_i}(L_n(T)(m), B_i)\), the collection \(\{\rho_i \circ \alpha_j, u_j\}_{j=1}^s\) is a \(P_i\)-coordinate system by restricting \(\rho_i\) to its image, we have \(\xi_{L_n(T_i)}^{\phi_i, B_i} = \xi_{L_n(T_i)}^{\Bar{\phi}_i, P_i}\). \cref{thmMiy} now gives the claim.
\end{proof}
\begin{remark}
    \cref{thmPmain} allows us to dispense with the technicalities that come from dealing with higher Zhu algebras. As we will see in \cref{secCat}, we can conceptualize this further by also replacing the modules \(L_n(T_i)\) with a common module \(G\), a projective generator of a suitable representation category of \(V\).
\end{remark}

\section{Trace spaces, symmetric functions and one-point functions}
\label{secCat}
\subsection{Pseudo-Traces and projective generators}
\label{subSecProj}
The goal of this section is to relate pseudo-trace functions defined on a specific class of modules of a suitable \(C_2\)-cofinite VOA \(V\), which form a representation category \(\Rep(V)\), to pseudo-traces defined on a projective generator \(G\) of \(\Rep(V)\). This will allow us to show that the map \(\phi \mapsto \xi_G^{\phi}\) defines an epimorphism \(\mathcal{S}(\End_V(G)) \rightarrow \mathcal{C}_1(V)\) (cf. \cref{propSC1iso}).

We have assumed in the preceding sections that \(V\) is a nonnegatively graded \(C_2\)-cofinite VOA such that any simple module has a nontrivial action of the Virasoro algebra. In the following, we are going to use several results from \cite{GR} and \cite{Hua}. The first reference assumes several extra conditions, for instance that \(V\) be of CFT-type, which we are not going to assume. For many results of the second reference, including those we work with explicitly, one needs to assume that \(V\) be \(C_1\)-cofinite. Under our assumption, this condition is already covered by \cref{lemC2C1}, so we do not have to impose additional conditions on \(V\) and keep those from \cref{conv}.

In \cite{Hua} and \cite{GR}, the category \(\Rep(V)\) is defined as the category of \textit{quasi-finite-dimensional} \(V\)-modules. These are logarithmic \(V\)-modules \(M = \bigoplus_{\lambda \in \mathbb{C}}M_{(\lambda)}\), that is, \(M\) is graded by generalized eigenspaces for \(L_0\), such that for all \(r \in \mathbb{R}\), one has the finiteness condition
\begin{align*}
    \sum_{\lambda \in \C, \Re(\lambda) \leq r} \dim_{\mathbb{C}}M_{(\lambda)} < \infty.
\end{align*}
We already know that any finitely generated weak module is logarithmic (cf. \cref{thmWisA}). Combining this with some results from \cite{Hua}, we can dispense with the notion of quasi-finite-dimensional modules and write down objects in \(\Rep(V)\) in a simpler fashion. 
\begin{proposition}
\label{propQuasi}
Every nonzero quasi-finite-dimensional module \(M \in \Rep(V)\) is finitely generated, and we can find complex numbers \(r_1, \hdots, r_k \in \mathbb{C}\) and integers \(d_1, \hdots, d_k \in \mathbb{N}\setminus \{0\}\) so that
\begin{align*}
M = \bigoplus_{i = 1}^k \bigoplus_{m = 0}^{\infty}M_{(r_i + m)},
\end{align*}
where \(M_{(r_i + m)} = \{w \in M|(L_0 - r_i - n)^{d_i}w = 0 \}\), \(M_{(r_i)} \neq 0\) and \(\dim_{\mathbb{C}}M_{(r_i + m)} < \infty\). Conversely, every finitely generated weak \(V\)-module is in \(\Rep(V)\).
\end{proposition}
\begin{proof}
    From \cref{thmWisA} we have that any finitely generated weak \(V\)-module is of the desired form. Recall that \(V\) is \(C_1\)-cofinite, there are finitely many conformal weights, and the Zhu algebras are finite-dimensional by \cref{propFMSM}. Now, using \cite[Proposition 4.3]{Hua}, the quasi-finite-dimensional \(V\)-modules are exactly the logarithmic modules of finite length and by \cite[Proposition 1.7]{Hua}, the logarithmic \(V\)-modules of finite length are finitely generated. Thus, the modules \(M \in \Rep(V)\) have the desired shape and we see that \(\sum_{\Re(\lambda) \leq r} \dim_{\mathbb{C}}M_{(\lambda)} < \infty\) is equivalent to \(\dim_{\mathbb{C}}M_{(r_i + m)} < \infty\). The converse is clear. 
\end{proof}

\begin{example}
    For any finitely generated \(A_n(V)\)-module \(X\), the admissible modules \(L_n(X)\) and \(\overline{M}_n(X)\) are in \(\Rep(V)\), since they are generated from \(L_n(X)(n) \cong X\) and \(\overline{M}_n(X)(n) \cong X\), respectively.
\end{example} 
The category \(\Rep(V)\) comes with an abelian structure.

\begin{theorem} [{\cite[Theorem 3.24, Proposition 4.3]{Hua}}]
\label{thmFA}
The category \(\Rep(V)\) is a \(\mathbb{C}\)-linear finite abelian category.
\end{theorem}
We follow \cite[Section 5]{GR}. Because of \cref{thmFA}, the category \(\Rep(V)\) has a projective generator \(G\). Writing \(E(G) := \End_V(G)\) or even shorter \(E = E(G)\) if the context is clear, the module \(G\) is a natural left \(E(G)\)-module, turning \(G\) into a \((V, E(G)^{\mathrm{op}})\)-bimodule by definition of the algebra \(E(G)\). For every projective generator \(G\), there is a functor
\begin{align*}
    \mathcal{F}_G\colon \Rep(V) \rightarrow \rmod E(G), \, \, \text{ where } \, \, \mathcal{F}_G = \Hom_V(G, -),
\end{align*}
with the right \(E(G)\)-module action on \(\mathcal{F}_G(M)\) for \(M \in \Rep(V)\) given by precomposition. This functor is part of an adjoint equivalence with pseudo-inverse given by the assignment
\begin{align*}
    \mathcal{H}_G\colon \rmod E(G) \rightarrow \Rep(V), \, \,\,\,\, \mathcal{H}_G = - \otimes_{E} G.
\end{align*}
For a right \(E\)-module \(W\), \(W \otimes_{E} G\) becomes a \(V\)-module by first equipping \(W \otimes_{\C}G\) with the \(V\)-action \(v_{(n)}(w \otimes g) = w \otimes v_{(n)}g\) and then showing that the action descends to the quotient \(W \otimes_{E} G\) by invoking the abelian structure. The projection map \(\pi_{\otimes} \colon W \otimes_{\C} G \rightarrow W \otimes_{E} G\) becomes a \(V\)-module map.

From \cite[Proposition 5.2, Remark 5.4]{GR}, we also record that \(G\) is a projective generator in the category \(E \text{-}\mathsf{Mod}\) of (not necessarily finite) left \(E\)-modules.

The reason for discussing the existence and properties of a projective generator is that we want to express pseudo-trace functions on possibly various different modules \(M \in \Rep(V)\) as pseudo-trace functions on a common module (the projective generator). The following proposition is from \cite[Remark 5.6]{GR}. Complementing the arguments contained in \cite[Proposition 5.5, Remark 5.6]{GR}, we give a full proof of the proposition.

\begin{proposition} [{\cite[Remark 5.6]{GR}}]
\label{propPGPT}
Let \(M \in \Rep(V)\). Let \(P \subseteq \End_V(M)\) be a subalgebra such that \(M\) is projective over \(P\) and \(\psi \in \mathcal{S}(P)\). Then, there is a \(\phi \in \mathcal{S}(E)\) such that \(\xi_G^{\phi, E} = \xi_M^{\psi, P}\). 
\end{proposition}
\begin{proof}
To define the pseudo-trace map \(\tr_G^{\phi}\), we give the maps \(\pi\colon E \otimes_{\mathbb{C}} G \rightarrow G, \pi(e \otimes g) = e(g)\) and \( \iota\colon G \rightarrow E \otimes_{\mathbb{C}} G\) such that \(\pi \circ \iota = \id_G\) (cf. \cref{lemPTCM}).\footnote{Note that \(G\) is certainly not a finite-dimensional vector space as demanded in \cref{lemPTCM}. However, as we are dealing with grade-preserving maps, for every generalized eigenspace component \(G_{(r+m)}\), we can simply restrict \(\pi\) to a map \(E \otimes_{\C} G_{(r+m)} \rightarrow G_{(r+m)}\), choose a section, and assemble this back into the map \(\iota\). Here and in the following, we tacitly assume that such an argument has been carried out whenever we mention possibly infinite-dimensional modules. This also applies to the graphical calculus, which we implicitly treat as a componentwise argument.} Consider the \((P, E)\)-bimodule \(\mathcal{F}(M) = \mathcal{F}_G(M)\). Writing \(r\colon \mathcal{F}(M) \otimes_{\mathbb{C}} E \rightarrow \mathcal{F}(M), r(f \otimes e) = f \circ e\) for the natural right action, we can define a map 
\begin{align*}
    p\colon \mathcal{F}(M) \otimes_{\mathbb{C}} G \rightarrow \mathcal{F}(M) \otimes_{\mathbb{C}} G, \, \, \, \, p = (r \otimes \id_G) \circ (\id_{\mathcal{F}(M)} \otimes \iota).
\end{align*}
This map satisfies \(p \circ (\id_{\mathcal{F}(M)} \otimes \pi - r \otimes \id_G) = 0\), from which it follows that it factors over \(\mathcal{F}(M) \otimes_{E} G\). In other words, there is a morphism \(\Bar{p}\colon \mathcal{F}(M) \otimes_E G \rightarrow \mathcal{F}(M) \otimes_{\mathbb{C}} G\) such that \(p = \Bar{p} \circ \pi_{\otimes}\). Since \(\pi_{\otimes} \circ p = \pi_{\otimes}\), this implies that \(\pi_{\otimes} \circ \Bar{p} = \id_{\mathcal{F}(M) \otimes_E G}\).

We now show that \(\mathcal{F}(M)\) is projective as a left \(P\)-module. Since \(G\) generates \(E\text{-}\mathsf{Mod}\), and \(E\) is a finite-dimensional \(E\)-module, there is a surjective \(E\)-morphism \(h \colon G^m \rightarrow E\), which induces an isomorphism \(\Bar{h}\colon G^m/\ker(h) \xrightarrow{\sim} E\). Since \(_{E}E\) is projective, \(\Bar{h}\) splits and \(G^m \cong E \oplus \ker(h)\). Because the \(V\)-isomorphism \(M \cong \mathcal{F}(M) \otimes_{E} G\) is natural in \(M\), in particular with respect to endomorphisms in \(P \subset \End_V(M)\), it is also a \(P\)-isomorphism. This, combined with the fact that \(\mathcal{F}(M) \otimes_{E} -\) defines an additive functor \(E\text{-}\mathsf{Mod}  \rightarrow P \text{-}\mathsf{Mod}\), implies
\begin{align*}
    M^m \cong (\mathcal{F}(M) \otimes_{E} G)^m \cong \mathcal{F}(M) \otimes_{E} G^m \cong \mathcal{F}(M) \oplus \mathcal{F}(M) \otimes_{E} \ker(h)
\end{align*}
as \(P\)-modules, which gives due to projectivity of \(M\) the claim. 

As a consequence, the \(P\)-module map \(\pi'\colon P \otimes_{\mathbb{C}} \mathcal{F}(M) \rightarrow \mathcal{F}(M)\) defined via the left \(P\)-action has a right inverse, that is \(\iota'\colon \mathcal{F}(M) \rightarrow  P \otimes_{\mathbb{C}} \mathcal{F}(M)\) such that \(\pi' \circ \iota' = \id_{\mathcal{F}(M)}\). With this, we define the linear map \(\phi\colon E \rightarrow \C\) by setting for every \(V\)-endomorphism \(e \in {E}\) 
\begin{align*}
    \phi(e) = \tr_{\mathcal{F}(M)}((\psi \otimes \id_{\mathcal{F}(M)}) \circ \iota' \circ r(- \otimes e)) \, \, = \, \, \adjincludegraphics[valign=c, scale = 0.9]{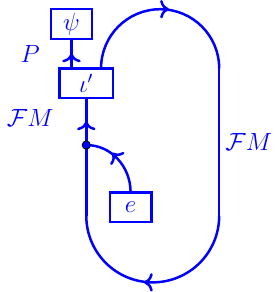},
\end{align*}
where we have denoted the right \(E\)-action given by the map \(r \colon \mathcal{F}(M) \otimes_{\C} E \rightarrow \mathcal{F}(M)\) by a dot on the strand labeled with \(\mathcal{F}(M)\). We have to show that \(\phi\) is a symmetric function on \(E\). Choose a \(P\)-coordinate system \( \{x_i, \alpha_i\}_{i = 1}^m\) of \(\mathcal{F}(M)\) compatible with \(\pi'\) and \(\iota'\). This means that the \(x_i\) form a basis of \(\mathcal{F}(M)\), and \(\alpha_i \in \Hom_P(\mathcal{F}(M), P)\) are such that \(\iota'(f) = \sum_{i = 1}^m\alpha_i(f) \otimes x_i\) for any \(f \in \mathcal{F}(M)\), see \cref{subsecPseudo} (here, \(u_j = x_j\) because of the definition of \(\pi'\)). 

Let now \(e, e' \in E\) and note that we can express \(f \circ e = \sum_{j = 1}^m\alpha_j(f \circ e) \circ x_j\). Setting \(f = x_i\), precomposing with \(e'\), applying \(\alpha_i\) and using \(P\)-linearity, we obtain the identity
\begin{align}
    \label{eqHelp}
    \alpha_i(x_i \circ e \circ e') = \sum_{j = 1}^m \alpha_j(x_i \circ e) \circ \alpha_i(x_j \circ e').
\end{align}
That \(\phi \in \mathcal{S}(E)\) now follows from an explicit computation of the trace, using the identity from \cref{eqHelp}. Denoting by \(x_i^* \in \mathcal{F}(M)^*\) the elements of the dual basis, we have
\begin{align*}
    \phi(e e') &= \tr_{\mathcal{F}(M)}((\psi \otimes \id_{\mathcal{F}(M)}) \circ \iota' \circ r(- \otimes e \circ e')) = \sum_{i,j = 1}^mx_j^*\left(\psi(\alpha_i(x_j \circ e \circ e'))x_i \right) \\
    &= \sum_{i = 1}^m \psi(\alpha_i(x_i \circ e \circ e')) \overset{(\text{\cref{eqHelp}})}{=} \sum_{i, j = 1}^m\psi(\alpha_j(x_i \circ e) \circ \alpha_i(x_j \circ e')).
\end{align*}
One obtains the same expression with \(e\) and \(e'\) interchanged for \(\phi(e' e)\). Using symmetry of \(\psi \in \mathcal{S}(P)\) and relabeling the summation indices, we see that indeed, \(\phi(ee') = \phi(e'e)\).

Moreover, from \(\pi'\) and \(\iota'\) we can define \(P\)-module maps
\begin{align*}
    \Tilde{\iota}\colon \mathcal{F}(M) \otimes_{E} G \xrightarrow{\Bar{p}} \mathcal{F}(M) \otimes_{\mathbb{C}} G \xrightarrow{\iota' \otimes \, \id_G} P \otimes_{\mathbb{C}} \mathcal{F}(M) \otimes_{\mathbb{C}} G \xrightarrow{\id_P \otimes \, \pi_{\otimes}}  P \otimes_{\mathbb{C}} \mathcal{F}(M) \otimes_{E} G \\
     \Tilde{\pi}\colon P \otimes_{\mathbb{C}} \mathcal{F}(M) \otimes_{E} G \xrightarrow{\id_P \otimes \, \Bar{p}}  P \otimes_{\mathbb{C}} \mathcal{F}(M) \otimes_{\mathbb{C}} G \xrightarrow{\pi' \otimes \, \id_G} \mathcal{F}(M) \otimes_{\mathbb{C}} G \xrightarrow{\pi_{\otimes}} \mathcal{F}(M) \otimes_{E} G,
\end{align*}
which give due to the isomorphism \(M \cong \mathcal{F}(M) \otimes_{E} G\) maps \(\Tilde{\pi}\colon P \otimes_{\mathbb{C}} M \rightarrow M, \Tilde{\iota}\colon M  \rightarrow P \otimes_{\mathbb{C}} M\), which we denote by abuse of notation by the same letter and which satisfy \(\Tilde{\pi} \circ \Tilde{\iota} = \id_M\).
We can express the pseudo-trace \(\tr_M^{\psi}\) with respect to these maps. We are now set for proving that \(\xi_M^{\psi, P} = \xi_G^{\phi, {E}}\), which we do using graphical calculus. For this, abbreviate \(\mathcal{O}^G = o(v)q^{L_0 - c_V/24} \in \End_E(G)\). Then,
\begin{align*}
    \xi_G^{\phi, E}(v, \tau) \, \, = \, \, \adjincludegraphics[valign=c, scale = 0.9]{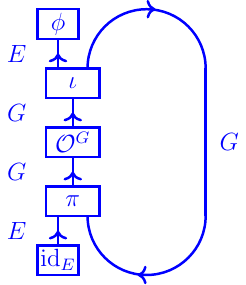} \,\, =  \adjincludegraphics[valign=c, scale = 0.9]{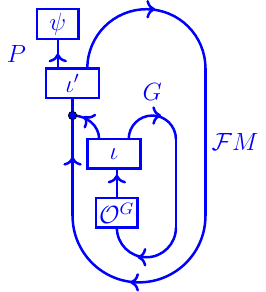} = \adjincludegraphics[valign=c, scale = 0.9]{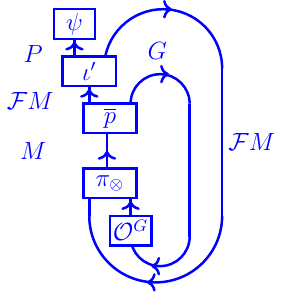}.
\end{align*}
In the first step, we substituted the definition of \(\phi\). In the second step, we used the definition of \(p\), the fact that \(p = \Bar{p} \circ \pi_{\otimes}\) and the identification \(M = \mathcal{F}(M) \otimes_E G\) as \(P\)-modules. We compute further 
\begin{align*}
    \adjincludegraphics[valign=c, scale = 0.9]{figures/step2.pdf} \, \, =  \, \, \adjincludegraphics[valign=c, scale = 0.9]{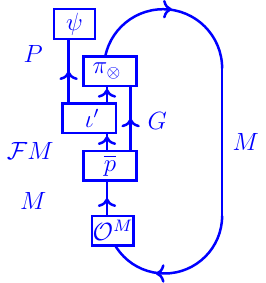} \, \, = \, \, \adjincludegraphics[valign=c, scale = 0.9]{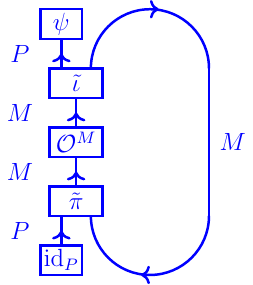} \, \, = \, \, \xi_M^{\psi, P}(v, \tau)
\end{align*}
where in the first step we used that \(\pi_{\otimes} \in \Hom_V(\mathcal{F}(M) \otimes_{\mathbb{C}} G, M)\) and applied the cyclicity of the trace map. Here, we used again the shorthand \(\mathcal{O}^M = o(v)q^{L_0 - c_V/24} \in \End_P(M)\). In the last step, we simply identified the composition of \(\pi_{\otimes}, \iota' \) and \(\Bar{p}\) as the map \(\Tilde{\iota}\). This identity holds independently of the choice of \(v \in V, \tau \in \mathbb{H}\) and thus concludes the proof. \qedhere
\end{proof}

Our toolkit is now complete to state and prove the main result of this section, proving a part of \cite[Conjecture 5.8]{GR}.
\begin{proposition}
\label{propSC1iso}
The linear map defined by \(\Phi_G\colon \mathcal{S}(E(G)) \rightarrow \mathcal{C}_1(V), \phi \mapsto \xi_G^{\phi, E}\) is an epimorphism.
\end{proposition}
\begin{proof}
We know from \cref{thmPmain} that there exist endomorphism algebras \(P_i\), symmetric functions \(\psi_i \in \mathcal{S}(P_i)\) and \((V, P_i)\)-bimodules \(L_n(T_i)\) with \(1 \leq i \leq k\) such that \(\mathcal{C}_1(V)\) is spanned by the pseudo-trace functions \(\xi_{L_n(T_i)}^{\psi_i, P_i}\). Using \cref{propPGPT}, for each \(\psi_i\) there is a \(\phi_i \in \mathcal{S}(E(G))\) such that \(\xi_G^{\phi_i, E} = \xi_{L_n(T_i)}^{\psi_i, P_i}\), so the linear map \(\Phi_G\) is surjective.
\end{proof}
\noindent It follows from \cref{propSC1iso} that \(\dim_{\mathbb{C}}\mathcal{C}_1(V) \leq \dim_{\mathbb{C}}S(E(G)) = \dim_{\mathbb{C}}(E(G)/[E(G), E(G)])\). As 
\begin{align}
    \label{eqDim}
    \dim_{\mathbb{C}}\mathcal{C}_1(V) = \dim_{\C} A_n(V)/[A_n(V), A_n(V)] - \dim_{\C} A_{n-1}(V)/[A_{n-1}(V), A_{n-1}(V)],
\end{align}
where the integer \(n \in \mathbb{N}\) is chosen such that \(n > n_a\),
we would like to relate \(E(G)\) to \(A_n(V)\) to establish the inverse inequality. In a special case, we will prove this inequality in the next section.

\subsection{The Gainutdinov-Runkel conjecture for separated conformal weights}
In addition to \(V\) being \(C_2\)-cofinite
and with nontrivial Virasoro action on simples, we assume subsequently that \(V\) has \textit{separated conformal weights}, specified in the following definition.
\begin{definition}
    A vertex operator algebra \(V\) with set of conformal weights \(\Lambda\) has \textit{separated conformal weights} if for any \(r,s \in \Lambda\) distinct, \(r - s \notin \mathbb{Z}\).
\end{definition}
\begin{example}
    \label{exWts}
    A class of examples for nonrational VOAs with separated conformal weights can be obtained by taking an appropriate rational VOA and tensoring it with a nonsemisimple local algebra. More precisely, let \(V\) be a rational and \(C_2\)-cofinite
    VOA satisfying that for \(r \neq s\) two conformal weights, \(r - s \notin \mathbb{Z}\). Let \(B\) be a finite-dimensional commutative local nonsemisimple algebra. It has a unique maximal ideal \(\mathfrak{m}\) and therefore a unique simple module \(B/\mathfrak{m}\). We can form the tensor product VOA \(V \otimes V^B = \oplus_{m \in \mathbb{N}} V_m \otimes_{\C} B\), with conformal vector given by \(\omega \otimes 1_B\), where \(\omega\) is the conformal vector of \(V\), see \cite[Section 2.5]{FHL}. Let us check why this VOA satisfies the assumptions in \cref{propEmb}:
    \begin{itemize}
        \item Since \(Y_{V \otimes V^B}(v \otimes b, z) = Y_V(v, z) \otimes (b \triangleright -)\), one has \(C_2(V) \otimes B \subset C_2(V \otimes B)\) and therefore \(V/C_2(V) \otimes B\) surjects onto \((V \otimes B)/C_2(V \otimes B)\). Hence, \(V \otimes V^B\) is \(C_2\)-cofinite.
        \item Because of \(C_2\)-cofiniteness, by \cite[Corollary 5.10]{Miy}, the conformal weights of \(V \otimes V^B\) are rational numbers. Then, by \cite[Proposition 4.7.2, Theorem 4.7.4]{FHL}, the irreducible modules are exactly the modules \(M \otimes \widehat{B/\mathfrak{m}}\), where \(M\) is an irreducible \(V\)-module and we used notation from \cref{exTr}. Thus, the conformal weights of \(V \otimes V^B\) are exactly those of \(V\), and accordingly no two distinct conformal weights are separated by integers.
        \item Finally, note that for \(M\) an irreducible \(V\)-module, for example, the \(V \otimes V^B\)-module \(M \otimes \widehat{B}\) is indecomposable, but not irreducible. Therefore, \(\Rep(V \otimes V^B)\) is not semisimple.
    \end{itemize}
    A concrete example is given, for instance, by \(V \otimes V^B = L_c \otimes \C[\varepsilon]/(\varepsilon^2)\) for \(c = 1/2\), where \(L_c\) is the Ising model (see, for instance, \cite[Section 7.2]{MaT}) with the set of conformal weights \(\Lambda = \{0, 1/2, 1/16\}\). In this case, \(L_c\) is rational and \(C_2\)-cofinite (and also of CFT-type).
    Note also that the Virasoro algebra acts nontrivially on all simple modules.
\end{example}
To prove that in this setting, \(\Phi_G \colon \mathcal{S}(E) \rightarrow \calc_1(V)\) is injective, we establish a series of intermediate results. Let us note that, although the separated-weights assumption is quite restrictive, the mechanisms underlying these results are likely to generalize.
\begin{lemma}
    \label{lemWt}
    Let \(M \in \Rep(V)\) be a (logarithmic) \(V\)-module. If \(V\) has separated conformal weights and \(r \in \Lambda\) is a conformal weight so that \(M_{(r + n)} \neq 0\) for some \(n \in \mathbb{N}\), then \(M_{(r + n)} \subset \Omega_n(M)\). 
\end{lemma}
\begin{proof}
    Let \(u \in V\) homogeneous so that \(\wt u - k - 1 < -n\) and let \(0 \neq w \in M_{(r + n)}\). If \(u_{(k)}w \neq 0\), then \(\wt(u_{(k)}w) = r + n + \wt u - k-1 < r\). If there was \(s \in \Lambda\) so that \(s + i = \wt(u_{(k)}w)\) for some \(i \geq 0\), then \(r - s \in \mathbb{Z}\), in contradiction to \(V\) having separated weights. Thus, \(u_{(k)}w = 0\), which implies by definition that \(w \in \Omega_n(M)\).
\end{proof}
We adapt a result of \cite{Hua}. For this, we call an \(A_n(V)\)-module \(X\) \textit{homogeneous of weight} \(\lambda \in \C\) if \(X\) is a generalized eigenspace with generalized eigenvalue \(\lambda\) for the operation of acting with \([\omega] \in A_n(V)\). Recall the integer \(n_b\) from \cref{subSecMiy}.
\begin{lemma}
    \label{lemHua}
    Let \(V\) be a VOA subject to the conditions from above, in particular the assumption of separated conformal weights. Then, the following hold:
    \begin{enumerate}[(i)]
        \item \label{claimI} For \(X\) be a homogeneous, projective and finite-dimensional \(A_n(V)\)-module of weight \(r + n\) for some \(r \in \Lambda\) and \(n \in \mathbb{N}\), the universal module \(\overline{M}_n(X)\) is projective in \(\Rep(V)\).
        \item \label{claimII} \looseness = -2 For \(n > n_b\), any simple \(V\)-module \(M\) of conformal weight \(r\) has a projective cover \( \overline{M}_n(X) \rightarrow M\), with \(X \subset A_n(V)_{(r + n)}\) a suitable homogeneous and projective ideal constructed in the proof.
    \end{enumerate}
\end{lemma}
\begin{proof}
    Since the results are minor modifications of the respective results in \cite{Hua}, we only present in detail the necessary adaptations. Note that we make some notational changes. Let us point out first that by \cite[Proposition 4.3]{Hua}, the category \(\Rep(V)\) and the category of finite-length generalized \(V\)-modules as defined in \cite{Hua} coincide.

    The proof of \labelcref{claimI} is similar to the proof of \cite[Theorem 3.22]{Hua}. Let \(M_1, M_2 \in \Rep(V)\) and \(f\colon \overline{M}_n(X) \rightarrow M_2\) and \(g \colon M_1 \rightarrow M_2\) be \(V\)-module maps so that \(g\) is surjective. To show projectivity, we need a map \(\tilde{f} \colon \overline{M}_n(X) \rightarrow M_1\) that satisfies \(g \circ \tilde{f} = f\). By \cref{lemWt}, we have \((M_i)_{(r+n)} \subset \Omega_n(M_i)\) as \(A_n(V)\)-submodules for \(i = 1, 2\). Using \(\overline{M}_n(X)(n) \cong X\), it follows that \(f(X) \subset (M_2)_{(r+n)}\) and \(g((M_1)_{(r+n)}) = (M_2)_{(r+n)}\). One can now argue as in the proof of \cite[Theorem 3.22]{Hua} and restrict the maps \(f\) and \(g\) to \(A_n(V)\)-module maps, then use projectivity of \(X\) to find an \(A_n(V)\)-module map \(X \rightarrow (M_1)_{(r + n)}\) and finally lift this to the desired map \(\tilde{f}\).

    We turn to the proof of \labelcref{claimII}. Let \(M\) be a simple \(V\)-module with conformal weight \(r\). As demonstrated in \cite[Lemma 3.5.1]{AN2}, for \(n > n_b\), the component \(M(n) = M_{r+n}\) is nonzero. By \cref{propOmega}, \(M(n)\) is a simple \(A_n(V)\)-module, and therefore \(r + n\) is a generalized eigenvalue for the operation of multiplying with \([\omega] \in A_n(V)\). By \cref{propFMSM}, \(A_n(V)\) is finite-dimensional, and by \cref{propGrad}, we have a decomposition of \(A_n(V)\) into generalized eigenspaces \(A_n(V)_{(r)}\) for \(r \in \Lambda(n)\), so by the arguments just presented, \(\Lambda + n \subset \Lambda(n)\).

    Now, arguing similarly as in the proof of \cite[Theorem 3.23]{Hua}, one finds a surjective module map \(\alpha \colon X \rightarrow M_{r+n}\) from an indecomposable \(A_n(V)\)-module that arises as a direct summand of \(A_n(V)_{(r+n)}\), and one shows that \(\alpha\) is a (homogeneous) projective cover of \(M_{r+n}\). By claim \labelcref{claimI}, \(\overline{M}_n(X)\) is projective and by \cref{thmMUniv}, there is a unique \(V\)-module map \(p \colon \overline{M}_n(X) \rightarrow M\) extending \(\alpha\). One shows as in the proof of \cite[Theorem 3.23]{Hua}, with the obvious replacements, that \(p\) defines a projective cover.
    \end{proof}

\begin{remark}
    Notice that the analogous results stated in \cite{Hua} depend crucially on choosing the integer \(n\) to be bigger than \(\mathrm{max_{r,s \in \Lambda}}|\Re(r) - \Re(s)|\), the maximal distance between the real parts of any two conformal weights. Because we assume \(V\) with separated conformal weights, we do not need this condition here. But, since we shift the component of the Zhu algebra in which we find the ideal we induce up to the desired projective cover by \(n\), we do in fact need to bound \(n\) from below by another lower bound \(n_b\) in order to force the desired Zhu algebra component to be nonzero.
\end{remark}
We are ready to prove the main result of this section, which establishes, in the setting of separated conformal weights, the inverse dimension estimate stated before \cref{eqDim}.
\begin{theorem}
    \label{propEmb}
    Let \(V\) be a \(C_2\)-cofinite
    VOA with separated conformal weights and nontrivial Virasoro action on all simple \(V\)-modules. Then, for every \(n > n_b\), there is a projective generator \(G_n\) of \(\Rep(V)\) so that the algebra \(E\left(G_n\right)^{\mathrm{op}}\) embeds as the nonunital subalgebra (and ideal) \(\oplus_{r \in \Lambda}A_n(V)_{(r+n)}\) in the higher Zhu algebra \(A_n(V)\). Moreover, for any projective generator \(G\), we have the inequality
    \begin{align*}
    \dim_{\C}\mathcal{S}(E(G)) \leq  \dim_{\C}A_n(V)/[A_n(V), A_n(V)] - \dim_{\C}A_{n-1}(V)/[A_{n-1}(V), A_{n-1}(V)].
    \end{align*}
\end{theorem}
\begin{proof}
    Because of the assumption \(n > n_b\), as explained in the proof of \cref{lemHua}, the Zhu algebra \(A_n(V)\) decomposes into generalized eigenspaces \(A_n(V) = \oplus_{r \in \Lambda(n)}A_n(V)_{(r)}\), where for the set of generalized eigenvalues of \(A_n(V)\) with respect to the central element \([\omega]\), we have \(\Lambda, \Lambda + n \subset \Lambda(n)\). This means that for every \(r \in \Lambda(n)\), \(([\omega] - r)^{d_r}A_n(V)_{(r)} = 0\) for some \(d_r \in \mathbb{N}\setminus\{0\}\). Decompose further \(A_n(V)_{(r)} = \oplus_{i = 1}^{k_r}X_{r,i}\) into indecomposable \(A_n(V)\)-modules. Then, every \(X_{r,i}\) is automatically projective. By \cref{lemHua} claim \labelcref{claimII} and its proof, for every irreducible \(V\)-module \(M\) with conformal weight \(r \in \Lambda\), there is an \(X_{r + n,i}\) for some \(1 \leq i \leq k_{r + n}\) from which one obtains a projective cover \(p_{r,i} \colon \overline{M}_n(X_{r+n,i}) \rightarrow M\). Denote the set of these \(X_{r+n,i}\) by \(\mathfrak{X}\). Then, \(G_n' = \oplus_{X \in \mathfrak{X}}\overline{M}_n(X)\) is a projective generator of \(\Rep(V)\) because the \(p_{r,i}\) corresponding to \(X \in \mathfrak{X}\) are projective covers of the collection of simple objects.
    Define moreover the \(V\)-module 
    \begin{align*}
        G_n = \oplus_{r \in \Lambda}\oplus_{j = 1}^{k_{r+n}}\overline{M}_n(X_{r+n,j}).
    \end{align*}
    Then \(G_n = G_n' \oplus Q\) for \(Q\) a projective module, since by \cref{lemHua} statement \labelcref{claimI}, every \(\overline{M}_n(X_{r+n,j})\) is projective, and therefore \(G_n\) is a projective generator as well. There is an isomorphism of algebras
    \begin{align*}
        E(G_n) &\cong \oplus_{r, s \in \Lambda}\oplus_{i = 1}^{k_{r+n}} \oplus_{j = 1}^{k_{s+n}} \Hom_V\left(\overline{M}_n(X_{r+n,i}), \overline{M}_n(X_{s+n,j})\right) \\
        &\cong \oplus_{r, s \in \Lambda}\oplus_{i = 1}^{k_{r+n}} \oplus_{j = 1}^{k_{s+n}} \Hom_{A_n(V)}\left(X_{r+n,i}, \Omega_n\overline{M}_n(X_{s+n,j})\right)
    \end{align*}
    by the universal property (cf. \cref{thmMUniv}). By \cite[Proposition 3.21]{Hua}, \(\overline{M}_n(X)(m) = \overline{M}_n(X)_{(r - n +m)}\) for homogeneous \(X\) of weight \(r\). Thus, for \(\varphi \in \Hom_{A_n(V)}\left(X_{r+n,i}, \Omega_n\overline{M}_n(X_{s+n,j})\right)\) and \(x \in X_{r+n,i}\),  
    \begin{align*}
        0 = \varphi(([\omega] - r -n)^{d_{r+n}} \triangleright x) = ([\omega] - r-n)^{d_{r+n}} \triangleright \varphi(x) = (L_0 - r-n)^{d_{r+n}}\varphi(x).
    \end{align*}
    By the condition on the conformal weights, \(\varphi(x) \neq 0\) only if \(r = s\), and then \(\varphi(x) \in \overline{M}_n(X_{r+n,j})(n) \cong X_{r+n,j}\). But there is also an inclusion \(\overline{M}_n(X_{r,i})(n) \subset \Omega_n\overline{M}_n(X_{r,i})\) by \cref{propOmega}, whence it follows that the nonzero hom-spaces in the above decomposition are isomorphic to \(\Hom_{A_n(V)}(X_{r+n,i}, X_{r+n,j})\). Since every \(X_{r,i}\) is projective and indecomposable, \(X_{r,i} = A_n(V)e_{r,i}\) for some primitive idempotent \(e_{r,i}\). The \(e_{r,i}\) for \(r \in \Lambda(n)\) and \(i \in \{1, \hdots, k_r\}\) form a complete set of primitive orthogonal idempotents of \(A_n(V)\) so that \(e_{s,j} *_n A_n(V)_{(r)} = A_n(V)_{(r)} *_n e_{s,j} = 0\) holds for \(s \neq r\). With this, we compute
    \begin{align*}
        E(G_n) &\cong \oplus_{r \in \Lambda} \oplus_{i,j = 1}^{k_{r+n}} \Hom_{A_n(V)}\left(X_{r+n,i}, X_{r+n,j}\right) \\ &\cong \oplus_{r \in \Lambda} \oplus_{i,j = 1}^{k_{r+n}}e_{r+n,i}A_n(V)_{(r+n)}e_{r+n,j} = \oplus_{r \in \Lambda} A_n(V)_{(r+n)}.
    \end{align*}
    The identification in the middle is an antiisomorphism of algebras, which proves the first part. 

    Let us denote by \(E_n = \oplus_{r \in \Lambda} A_n(V)_{(r+n)}\) the ideal in \(A_n(V)\) that is the image of the embedding of \(E(G_n)\). Let moreover \(K_n = O_{n-1}(V)/O_n(V)\) denote the kernel of the map \(\pi_n \colon A_n(V) \rightarrow A_{n-1}(V)\). We claim that \(E_n \subset K_n\). Indeed, assuming that \(x \in A_n(V)_{(r+n)}\) where \(r \in \Lambda\), one has \(0 = \pi_n(([\omega] - r - n)^{d_{r+n}}x) = ([\omega] - r - n)^{d_{r+n}}\pi_n(x)\) since \(\pi_n\) is induced by the identity map \(\id_V\) that maps the central element \(\omega\) to itself. One therefore has \(\pi_n(x) \in A_{n-1}(V)_{(r + n)}\). However, if the latter space were nonzero, then because \(\Lambda(n-1) \subset \{s + i| s \in \Lambda, 0 \leq i \leq n-1\}\) by \cref{propGrad}, the condition that the conformal weights are separated would be violated. It follows that \(x \in K_n\). 

    Now, the algebra map \(\pi_n\) induces a well-defined surjective map \(\overline{\pi}_n \colon A_n(V)/[A_n(V), A_n(V)] \rightarrow A_{n-1}(V)/A_{n-1}(V), A_{n-1}(V)]\) with kernel \(\overline{K}_n\) of dimension
    \begin{align*}
        \dim_{\C}\overline{K}_n = \dim_{\C}A_n(V)/[A_n(V), A_n(V)] - \dim_{\C}A_{n-1}(V)/[A_{n-1}(V), A_{n-1}(V)].
    \end{align*}
    Now, note that \(E_n \cap [A_n(V), A_n(V)] = [E_n, E_n]\), which follows from the fact \(A_n(V)_{(r)} *_n A_n(V)_{(s)} = 0\) for \(s \neq r\), so \(E_n/[E_n, E_n]\) gives a well-defined subspace of \(\overline{K}_n\).
    The desired inequality follows by noting that \(\dim_{\C}\mathcal{S}(E(G_n)) = \dim_{\C}E_n/[E_n, E_n]\) for the projective generator \(G_n\). Finally, for any projective generator \(G\) of \(\Rep(V)\), there is an equivalence of categories provided by
    \begin{align*}
        \mathcal{F}_G \circ \calh_{G_n}\colon \rmod E(G_n) \rightarrow \rmod E(G)
    \end{align*}
    using the notation from \cref{subSecProj}. Thus, the algebras \(E(G)\) and \(E(G_n)\) are Morita equivalent. Since their commutator quotients coincide with the zeroth Hochschild homology \(HH_0(E(G_n))\) and \(HH_0(E(G))\), respectively, and Hochschild homology is Morita invariant, the dimension of the two spaces of symmetric functions coincide.
\end{proof}
We obtain \cite[Conjecture 5.8]{GR} for the case of separated conformal weights.
\begin{corollary}
    \label{corIso}
    For \(V\) a \(C_2\)-cofinite vertex operator algebra 
    with nontrivial action on its simples by the Virasoro algebra and with separated conformal weights, the linear map \(\Phi_G \colon \mathcal{S}(E) \rightarrow \calc_1(V)\) is an isomorphism.
\end{corollary}
\begin{proof}
    This follows directly from \cref{propSC1iso}, \cref{eqDim} and \cref{propEmb} for \(n > n_a\).
\end{proof}

\subsection{Trace spaces of additive categories}
\label{subSecTr}
Let \(A\) be a finite-dimensional associative algebra.
Define the \textit{trace space of the algebra} \(A\), the vector space \(\mathrm{Tr}(A)\), as
\begin{align*}
    \mathrm{Tr}(A) = A/[A, A].
\end{align*}
The trace space fits into the exact sequence of vector spaces
\begin{align}
\label{eqTrAlg}
A \otimes_{\mathbb{C}} A \longrightarrow A \longrightarrow \mathrm{Tr}(A) \longrightarrow 0,
\end{align}
where the first map is given by \(a \otimes b \mapsto [a, b] = ab - ba\) and the second is the canonical projection. The trace space is relevant to us since the symmetric functions on \(A\) can be described by its dual, that is \(\mathcal{S}(A) \cong \Tr(A)^*\), which follows from elementary arguments. 

\medskip
We are interested in describing the symmetric functions on the algebra \(E = \End_V(G)\) through the representation category of \(V\). We are therefore seeking a categorical analogue of the trace space. For the following definition, we need the concept of a coend (see, for instance, \cite{L}).

\begin{definition}
    For an essentially small \(\C\)-linear category \(\mathcal{A}\), we define the \textit{trace space of} \(\mathcal{A}\) as the vector space given by the coend of the functor \(\Hom_{\mathcal{A}}(-,-)\colon \mathcal{A}^{\mathrm{op}} \times \mathcal{A} \rightarrow \vect_{\C}\), that is
    \begin{align*}
        \Tr(\mathcal{A}) = \int^{X \in \mathcal{A}} \Hom_{\mathcal{A}}(X, X).
    \end{align*}
\end{definition}
Note that the trace space always exists because the category of vector spaces is cocomplete. It follows from the universal property of the coend that a functor \(F \colon \cala \rightarrow \mathcal{B}\) induces a linear map \(\Tr(F) \colon \Tr(\cala) \rightarrow \Tr(\mathcal{B})\), and this assignment extends to a functor \(\Tr \colon \mathsf{Cat}_{\C} \rightarrow \vect_{\C}\) from the category of essentially small \(\C\)-linear categories to vector spaces. If \(F\colon \mathcal{A \simeq \mathcal{B}}\) is an equivalence of \(\C\)-linear categories, then \(\Tr(F)\) is an isomorphism \(\Tr(\mathcal{A}) \cong \Tr(\mathcal{B})\). 
\begin{remark}
    We explain the analogy between the notion of trace space for an essentially small \(\C\)-linear category \(\mathcal{A}\) and the one for an algebra \(A\). Recall that for a functor \(F\colon \mathcal{C}^{\mathrm{op}} \times \mathcal{C} \rightarrow \mathcal{D}\) into a cocomplete category \(\mathcal{D}\), the coend can be computed explicitly as a coequalizer,
    \begin{align*}
        \int^{c \in \mathcal{C}}F(c, c) \cong \mathrm{coeq}\left( \bigoplus_{c, c' \in \calc, f \in \mathcal{C}(c, c')}F(c', c) \, \substack{\overset{F^*}{\textstyle \longrightarrow}  \\ \underset{F_*}{\textstyle \longrightarrow}} \, \bigoplus_{c \in \mathcal{C}}F(c, c)\right).
    \end{align*}
    For \(\calc = \mathcal{A}\), \(\cald = \vect_{\C}\) and \(F = \Hom_{\mathcal{A}}(-,-)\), the two arrows \(F^*, F_*\) are given on a fixed component \(\mathcal{A}(X, Y)\) and a chosen \(f \in \mathcal{A}(Y, X)\) by pre- and postcomposition, respectively. Hence, we have
    \begin{align*}
        \Tr( \mathcal{A}) \cong \left( \bigoplus_{X \in \mathcal{A}}\mathcal{A}(X, X) \right)/I(\mathcal{A}) \, \, \, \, \text{where} \,\,\, I(\mathcal{A}) = \langle g \circ f - f \circ g| g \in \mathcal{A}(X, Y), f \in \mathcal{A}(Y, X)\rangle_{\mathbb{C}}.
    \end{align*}
    In analogy to the exact sequence from \cref{eqTrAlg}, we observe that \(\Tr(\mathcal{A})\) fits into the sequence
    \begin{align*}
        \bigoplus_{X, Y \in \mathcal{A}}\mathcal{A}(X, Y) \otimes_{\mathbb{C}} \mathcal{A}(Y, X) \xrightarrow{\varphi} \bigoplus_{X \in \mathcal{A}}\mathcal{A}(X, X) \rightarrow \Tr(\mathcal{A}) \rightarrow 0,
    \end{align*}
    where the first map is given by \(\varphi(f \otimes g) = f \circ g - g \circ f\) and the second is the obvious projection. The latter description of the trace space is the viewpoint adopted in \cite[Section 2]{K}.
\end{remark}
For the \(\C\)-linear category of finitely generated projective \(A\)-modules \(\mathcal{A} = \rproj A\), we will give an explicit description to the \textit{dual} of the trace space by providing an explicit sequence of isomorphisms
\begin{align*}
    \Tr( \rproj A)^* \cong \mathrm{Cyc}(\rproj A) \cong \mathcal{S}(A)
\end{align*}
in the subsequent section,
where we defined the middle space in a moment. This also implies that \(\Tr(\rproj A) \cong \Tr(A)\). By substituting \(A = E(G)\), we can then formulate our main result.

\subsection{Trace space dual and one-point functions}
\looseness = -1 For an essentially small \(\mathbb{C}\)-linear category \(\mathcal{A}\), consider an arbitrary cowedge \(\beta\colon \Hom_{\mathcal{A}}(-,-) \xRightarrow{\bullet} \Delta_{\mathbb{C}}\), with \(\Delta_{\C} \colon \cala \times \cala^{\mathrm{op}} \rightarrow \vect_{\C}\) the constant functor with value the one-dimensional vector space \(\C\). This means that we have a commutative diagram
    \[
    \begin{tikzcd}
        \mathcal{A}(X, Y) \ar[d, swap, "{\mathcal{A}(f, Y)}"] \ar[r, "{\mathcal{A}(X, f)}"] & \mathcal{A}(X, X) \ar[d, "{\beta_X}"] \\
        \mathcal{A}(Y, Y) \ar[r, swap, "{\beta_Y}"] & \mathbb{C}
    \end{tikzcd}
    \]
which says that \(\beta_X(f \circ g) = \beta_Y(g \circ f)\) for every \(f \in \mathcal{A}(Y, X), g \in \mathcal{A}(X, Y)\) and \(X, Y \in \mathcal{A}\). In other words, a cowedge \(\beta\) corresponds to a family of homomorphisms \(\beta_{\bullet} = \{\beta_X\}_{X \in \mathcal{A}}\) satisfying the cyclicity property. We denote the \(\mathbb{C}\)-vector space of these objects with componentwise addition by
    \begin{align*}
        \mathrm{Cyc}(\mathcal{A}) = \left\{ \, \beta_{\bullet} \, \Bigg | \, \begin{array}{l}
             \beta_X \in \Hom_{\mathbb{C}}(\mathcal{A}(X, X),  \mathbb{C}), \, \beta_X(f \circ g) = \beta_Y(g \circ f) \, \\ 
             \text{for all} \, X, Y \in \mathcal{A}, \text{ and } f \in \mathcal{A}(Y, X), g \in \mathcal{A}(X, Y)
        \end{array} \right\}.
    \end{align*}
We have the following characterization of this space.
    \begin{lemma}
    \label{lemTrW}
        There is a canonical isomorphism of vector spaces \(\Upsilon \colon \mathrm{Tr}(\mathcal{A})^* \xrightarrow{\sim} \mathrm{Cyc}(\mathcal{A})\).
    \end{lemma}
    \begin{proof}
    Because \(\mathrm{Tr}(\mathcal{A})\) is a coend with defining cowedge \(\omega \colon \Hom_{\mathcal{A}}(-,-) \xRightarrow{\bullet} \Delta_{\int^{X \in \mathcal{A}}\mathcal{A}(X, X)}\), for every \(\beta_{\bullet} \in \mathrm{Cyc}(\mathcal{A})\), there is a unique linear morphism \(h^{\beta} \colon \mathrm{Tr}(\mathcal{A}) \rightarrow \mathbb{C}\) making the diagram
    \[\begin{tikzcd}[ampersand replacement=\&]
	\& {\mathrm{Tr}(\mathcal{A})} \\
	{\mathrm{Hom}_{\mathcal{A}}(X, X)} \& {\mathbb{C}}
	\arrow["{ \exists !h^{\beta}}", dashed, from=1-2, to=2-2]
	\arrow["{\omega_X}", shift left, curve={height=-6pt}, from=2-1, to=1-2]
	\arrow["{\beta_X}"', from=2-1, to=2-2]
    \end{tikzcd}\]
    commute. Define the \(\C\)-linear map \(\Upsilon\) by \(\Upsilon(h) = (h \circ \omega_{X})_{X \in \mathcal{A}}\).
    The cyclicity of \(\omega\) shows that the map is well-defined and defines an isomorphism because by the commutativity of the above diagram, every family \(\beta_{\bullet} \in \mathrm{Cyc}(\mathcal{A})\) has the unique preimage \(h^{\beta}\).
    \end{proof}
We specialize to \(\mathcal{A} = \rproj A\) for \(A\) a finite-dimensional algebra and denote by \(\mu\colon A \rightarrow \End_A(A_A)\) the bijective mapping \(a \mapsto a \triangleright \id_A\), where the left action of \(A\) on \(\End_A(A_A)\) is given by \((a \triangleright f)(b) = a \cdot f(b)\) for \(a, b \in A, f \in \End_A(A_A)\). In this setting, another characterization of \(\mathrm{Cyc}(\mathcal{A})\) is available.

\begin{lemma}
    \label{lemWisoS} With the notation above, the linear map \( \Theta \colon \mathrm{Cyc}(\rproj A) \rightarrow \mathcal{S}(A), \beta_{\bullet} \mapsto \beta_A \circ \mu\) is an isomorphism of vector spaces with inverse \(\phi \mapsto \tr_{\bullet}^{\phi}\). 
    
\end{lemma}
\begin{proof}
        For any \(\beta_{\bullet} \in \mathrm{Cyc}(\rproj A)\) one has \(\beta_A \circ \mu \in \mathcal{S}(A)\) by definition, so the map \(\Theta\) is well-defined. The regular module \(A\) is a projective generator for the category of finitely generated right \(A\)-modules \(\rmod A\), so we find for any projective object \(X \in \rproj A\) an integer \(m \in \mathbb{N}\) and a surjective right \(A\)-morphism \(p \colon A^m \twoheadrightarrow X\). Using projectivity of \(X\), for an arbitrary \(A\)-module map \(g \in \End_AX\) there is a lift \(\Tilde{g} \colon X \rightarrow A^m\) so that \(g\) factors through \(\Tilde{g}\) and \(p\), that is, the diagram
        \[
            \begin{tikzcd}
            & A^m \ar[d, "{p}"] \\
            X \ar[ur, dotted, "{\Tilde{g}}"] \ar[r, "{g}"] & X
        \end{tikzcd}
        \]
        commutes. With this factorization at hand, any component \(\beta_X\) can be expressed as
        \begin{align*}
            \beta_X(g) = \beta_{A^m}(\Tilde{g} \circ p) = \beta_{A^m}\left(\sum_{i = 1}^m \iota_i \circ \pi_i \circ \Tilde{g} \circ p \right) = \sum_{i = 1}^m\beta_A(\pi_i \circ \Tilde{g} \circ p \circ \iota_i),
        \end{align*}
        with the canonical injections \(\iota_i\colon A \rightarrow A^m\) and projections \(\pi_i\colon A^m \rightarrow A\) for the \(i\)'th component. Here, we have used cyclicity and the decomposition \(\sum_{i=1}^m \iota_i \circ \pi_i = \id_{A^m}\) of the identity map on \(A^m\). We therefore see that \(\beta_{\bullet}\) is completely determined by its component \(\beta_A\) and the cyclicity property. 
        
        Note that the maps \(\iota_i, \pi_i\) yield an \(A\)-coordinate system for \(A^m\), that is, for every vector \((a_1, \hdots, a_m) \in A^m\), we can write \((a_1, \hdots, a_m) = \sum_{i=1}^m\iota_i(1_A) \triangleleft \pi_i(a_1, \hdots, a_m)\). We also see that for any \(\phi \in \mathcal{S}(A)\), the system of pseudo-traces \(\tr_{\bullet}^{\phi}\) is an element of \(\mathrm{Cyc}(\rproj A)\). By definition of the pseudo-trace,
        \begin{align*}
            \tr_{X}^{\, \beta_A \, \circ \, \mu}(g) = \tr_{A^m}^{\, \beta_A \, \circ \, \mu}(\Tilde{g} \circ p) = \sum_{i=1}^m\beta_A((\pi_i \circ \Tilde{g} \circ p \circ \iota_i)(1_A) \triangleright \id_A)
        \end{align*}
        for the symmetric function \(\beta_A \circ \mu\) and thus, \(\beta_X = \tr_X^{\beta_A \circ \mu}\). This shows that the map \(\phi \mapsto \tr_{\bullet}^{\phi}\) defines a left inverse to \(\Theta\).
        
        For the other direction, because of the isomorphism \(\mu\colon \End_A(A_A) \cong A\), any \(A\)-endomorphism \(g\) can be written as a right multiplication \(g = a \triangleright \id_A\) for a unique \(a \in A\). Choosing the \(A\)-coordinate sytem \(\alpha = \id_A, u = 1_A\) reveals that
        \begin{align*}
            \tr_A^{\phi}(g) = \phi(\alpha(g(u))) = \phi(g(1_A)) = \phi(a).
        \end{align*}
        Hence, it follows that \(\Theta(\tr_{\bullet}^{\phi}) = \tr_A^{\phi} \circ \, \mu = \phi\), finishing the proof.
\end{proof}

We summarize the results of this article in the following theorem. Recall the notation \(E = E(G) = \End_V(G)\) for \(G\) a projective generator of the category \(\Rep(V)\), and the equivalence of categories \(\mathcal{H}_G \colon \rproj E(G) \rightarrow \Rep(V)\). Moreover, denote by \(\omega_E \colon \End_E(E) \rightarrow \Tr(\rproj E)\) the component at \(E\) of the cowedge \(\omega\) defining the coend.
\begin{theorem}
\label{thmMain}
    Let \(V = \bigoplus_{m = 0}^{\infty}V_m\) be a 
    \(C_2\)-cofinite vertex operator algebra such that every simple \(V\)-module carries a nontrivial action of the Virasoro algebra. Then, for every projective generator \(G \in \Rep(V)\), there is a canonical epimorphism of finite-dimensional \(\mathbb{C}\)-vector spaces 
    \begin{align*}
        \Gamma_G\colon \mathrm{Tr}(\proj \Rep(V))^* \twoheadrightarrow \mathcal{C}_1(V), \, \, \, \, \, \,\Gamma_G(h) = \xi_G^{\phi_h},
    \end{align*}
    where \(\phi_h \in \mathcal{S}(E)\) is the symmetric function defined as the composition \(\phi_h = \Tr(\mathcal{H}_G)^*(h) \circ \omega_E \circ \mu\).
    If additionally \(V\) has separated conformal weights, then \(\Gamma_G\) is an isomorphism.
\end{theorem}
\begin{proof}
    The theorem is a simple corollary of the results established in this section. First, note that an equivalence of categories preserves and reflects projectives, so that \(\proj( \Rep(V)) \simeq \rproj E(G)\) for every projective generator \(G\). We then have the following chain of linear morphisms
    \begin{align*}
        \mathrm{Tr}(\proj \Rep(V))^* \xrightarrow{\Tr(\mathcal{H}_G|_{\rproj E})^*} \mathrm{Tr}(\rproj E(G))^* \xrightarrow{\Upsilon} \mathrm{Cyc}(\rproj E(G)) \xrightarrow{\Theta} \mathcal{S}(E(G)) \xrightarrow{\Phi_G} \mathcal{C}_1(V)
    \end{align*}
    composing to \(\Gamma_G\), where the the second and third maps come from \cref{lemTrW} and \cref{lemWisoS}, respectively, by putting \(\mathcal{A} = \rproj E(G)\), and all of the first three maps are isomorphisms. It thus follows that the dimension of the dual of the trace space is finite,
\begin{align*}
    \dim_{\mathbb{C}}\mathrm{Tr}(\proj \Rep(V))^* = \dim_{\mathbb{C}}E(G)/[E(G), E(G)] < \infty.
\end{align*}
    The last map \(\Phi_G\) is the epimorphism from \cref{propSC1iso}, which shows that the concatenated map is surjective as well. In the case that \(V\) has separated conformal weights, \(\Phi_G\) is an isomorphism by \cref{corIso}, rendering \(\Gamma_G\) an isomorphism identifying \(\mathrm{Tr}(\proj \Rep(V))^* \cong \calc_1(V)\). 
\end{proof}
\begin{remark}
In the language of \cite[Definition 5.1.6]{KL}, for any \(\C\)-linear finite abelian category \(\cala\), the set of objects of the subcategory \(\proj(\cala) \subset \cala \) \(p\)-generates \(\cala\). This is because for any object \(X \in \cala\) and a projective generator \(G\) of \(\cala\), there is an epimorphism \(p_X \colon G^m \rightarrow X\) for some \(m \in \mathbb{N}\) and for any projective \(P \in \proj(\cala) \), any morphism \(f \colon P \rightarrow X\) factors through \(p_X\) by projectivity of \(P\). It therefore seems tempting to try to apply \cite[Proposition 5.1.7]{KL} to identify the coends \(\Tr(\cala)\) and \(\Tr(\proj(\cala))\) in the context \(\cala = \Rep(V)\). This, however, is not possible, because the result requires the considered bifunctor to be exact in each argument, whereas \(\Hom_V(-,-)\) is not since \(V\) is not necessarily rational.
\end{remark}
As noted in \cite[Remark 5.13]{GR}, the morphism \(\Phi_G\) is an isomorphism when \(V\) is rational, and so \(\Gamma_G \colon \mathrm{Tr}(\proj \Rep(V))^* \rightarrow \mathcal{C}_1(V)\) is as well. We briefly discuss a couple of nonrational examples.
\begin{example}
    \begin{enumerate}[(i)]
        \item For \(V^B\) as in \cref{exComm} and \(B\) possibly nonsemisimple, we had seen in \cref{exTr} that \(\calc_1(V^B)\) cannot necessarily be spanned by trace functions. However, note that mapping a finitely generated \(B\)-module \(W\) to \(\widehat{W}\) (cf. \cref{exTr}) induces an isomorphism of categories \(\rmod B \rightarrow \Rep V^B\) mapping the projective generator \(B\) of \(\rmod B\) to the projective generator \(\widehat{B}\) of \(\Rep V^B\). Now \(\End_{V^B}(\widehat{B}) \cong B\), so indeed, \(\mathcal{S}(\End_{V^B}(\widehat{B})) \cong B^* \cong \calc_1(V^B)\). However, note that the Virasoro action on \(V^B\) is trivial, so the VOA does not quite satisfy the assumptions for \cref{thmMain}. 
        \item Using the notation of the first example, as pointed out in \cref{exWts}, for \(B\) local, one can find a suitable rational VOA \(V\) so that the nonrational VOA \(V \otimes V^B\) satisfies all the assumptions of \cref{thmMain}, including that of separated conformal weights. 
        \item As observed in \cite{GR} after stating Conjecture 5.8, it is proved in \cite[Corollary 3.1]{AM} for the triplet VOA \(W_{1,p}\) that the space spanned by generalized characters is \((3p - 1)\)-dimensional. This agrees with the dimension of \(\mathcal{S}(E)\) for any choice of projective generator \(G\) of \(\mathrm{Rep}(W_{1,p})\). By \cref{propPGPT}, the space of generalized characters is spanned by functions \(\xi_G^{\phi, E}(\mathbf{1}, -)\). Thus, the linear map \(\Phi_G \colon \mathcal{S}(E) \rightarrow \calc_1(W_{1,p})\) is injective and hence, by \cref{propSC1iso}, an isomorphism. 
    \end{enumerate}
\end{example}

\end{document}

%% file: figures/style1.tikzstyles
\tikzstyle{new style 0}=[fill=white, draw=black, shape=circle]
\tikzstyle{new style 1}=[fill=white, draw=black, shape=rectangle]
\tikzstyle{wide rect}=[fill=white, draw=black, shape=rectangle, minimum width=1.25cm, minimum height=0.5cm]
\tikzstyle{hw}=[fill=white, draw=black, shape=rectangle, minimum height=0.5cm, minimum width=1.0cm]

\tikzstyle{new edge style 1}=[->]
\tikzstyle{new edge style 0}=[<-]

%% file: main.bbl
\begin{thebibliography}{9}
\thispagestyle{empty}

\bibitem[A]{A}
Y. Arike (2010), \emph{Some remarks on symmetric linear functions and pseudotrace maps}. Proc. Japan Acad. Ser. A Math. Sci. 86(7).

\bibitem[AM]{AM}
D. Adamovic, A. Milas (2009), \emph{An analogue of modular BPZ-equation in logarithmic (super)conformal field theory}. Vertex operator algebras and related areas, Contemp. Math. 497 1–17.

\bibitem[AN]{AN}
Y. Arike, K. Nagatomo (2013), \emph{Some remarks on pseudo-trace functions for orbifold models associated with symplectic fermions}. International Journal of Mathematics
Vol. 24, No. 2 (2013) 1350008.

\bibitem[AN2]{AN2}
Y. Arike, K. Nagatomo (2010), \emph{One-point functions over elliptic curves}. arXiv:1008.3771.

\bibitem[Br]{Br}
M. Broué (2009), \emph{Higman's criterion revisited}. Michigan Math. J. 58.

\bibitem[Bo]{Bo}
R. Borcherds (1986), \emph{Vertex algebras, Kac-Moody algebras, and the monster}. Proc. Natl. Acad. Sci. USA 83, 3068-307 J.

%
\bibitem[DLM1]{DLM1}
C. Dong, H. Li, G. Mason (1998), \emph{Twisted representations of vertex operator algebras}. Math Ann 310, 571–600.

\bibitem[DLM2]{DLM2}
C. Dong, H. Li, G. Mason (1998), \emph{Vertex operator algebras and associative algebras}.	Journal of Algebra 206, 67-96.

\bibitem[DLM3]{DLM3}
C. Dong, H. Li, G. Mason (2000), \emph{Modular-invariance of trace functions in orbifold theory and generalized moonshine}. Commun. Math. Phys. 214, 1–56.

%

\bibitem[FHL]{FHL}
I. Frenkel, Y.-Z. Huang, J. Lepowsky (1993), \emph{On axiomatic approaches
to vertex operator algebras
and modules}. Memoirs of the American Mathematical Society, 494.

\bibitem[FLM]{FLM}
I. Frenkel, J. Lepowsky and A. Meurman (1988), \emph{Vertex Operator Algebras and the Monster}. Pure and Applied Math., Vol. 134, Academic Press, Boston.

\bibitem[GR]{GR}
A. Gainutdinov, I. Runkel (2018), \emph{The non-semisimple Verlinde formula and pseudo-trace functions}. Journal of pure and applied algebra 223(2), 660 - 690.

\bibitem[GZ]{GZ}
B. Gui, H. Zhang (2025), \emph{How are pseudo-q-traces related to (co)ends?}. arXiv:2508.04532 [math.QA].

\bibitem[Hua]{Hua}
Y.-Z. Huang (2009), \emph{Cofiniteness conditions, projective covers and the logarithmic tensor product theory}. Journal of Pure and Applied Algebra 213, 458–475.

%
\bibitem[K]{K}
B. Keller (2021), \emph{Hochschild (Co)homology and Derived Categories}. Bulletin of the Iranian Mathematical Society volume 47, 57–83.

\bibitem[KL]{KL}
T. Kerler, V. Lyubashenko (2001), \emph{Non-Semisimple Topological Quantum Field Theories for 3-Manifolds with Corners}. Lecture Notes in Mathematics (Vol. 1765).

\bibitem[L]{L}
F. Loregian (2021), \emph{Coend calculus}. London Mathematical Society, Lecture Note Series 468.

\bibitem[Li]{Li}
H. Li (1999), \emph{Some finiteness properties of regular vertex operator algebras}. J. Alg. 212, 495–514.

\bibitem[LL]{LL}
J. Lepowsky, H. Li (2004), \emph{Introduction to Vertex Operator Algebras and Their Representations}. Progress in Mathematics 227, Birkhäuser Boston, MA.

\bibitem[MaT]{MaT}
G. Mason, M. Tuite (2018), \emph{Vertex Operators and Modular Forms}. A Window into Zeta and Modular Physics, ed Kirsten, K. and Williams, F., MSRI Publications 57.


\bibitem[Miy]{Miy}
M. Miyamoto (2004), \emph{Modular invariance of vertex operator algebras satisfying 
C2-cofiniteness}. Duke Math. J. 122(1), 51-91.

\bibitem[NS]{NS}
C. Nesbitt and W. Scott (1943), \emph{Some remarks on algebras over an algebraically
closed field}. Ann. of Math. 44, No.3, 534–553.

\bibitem[S]{S}
M.-N. Steffen (2022), \emph{Vertex operator algebras and pseudo-trace functions}. Master's thesis, University of Hamburg.

\bibitem[TV]{TV}
V. Turaev and A. Virelizier (2017), \emph{Monoidal categories and topological field theory}. Progress in Mathematics, Vol. 322, Birkhäuser.

\bibitem[Z]{Z}
Y. Zhu (1996), \emph{Modular invariance of characters of vertex operator algebras}. J. Amer. Math. Soc. 9(1), 237-302.
\thispagestyle{empty}
\end{thebibliography}
